\documentclass[10pt]{amsart}

\usepackage{amssymb}      
\usepackage{amsmath}      
\usepackage{amsthm}      
\usepackage{amsbsy}      
\usepackage{amscd}      
\usepackage[dvips]{graphicx}

\theoremstyle{plain}      
\newtheorem{theorem}{Theorem}[section]           
\newtheorem{lemma}{Lemma}[section]      
\newtheorem{corollary}{Corollary}[section]      
\newtheorem{proposition}{Proposition}[section]      
\newtheorem{conjecture}{Conjecture}[section]

\theoremstyle{remark}      
\newtheorem{remark}{Remark}[section]

\newtheorem{example}{Example}[section]

\newcommand{\Q}{{\mathbb{Q}}}        
\newcommand{\Z}{{\mathbb{Z}}}   
   
\newcommand{\C}{{\mathbb{C}}}      
\newcommand{\R}{{\mathbb{R}}}

 \newcommand{\M}{{\mathcal{M}}}

\newcommand{\ldual}[1]{#1^{*}}

\newcommand{\X}{\mathcal X}
\newcommand{\A}{\mathfrak A}
\newcommand{\G}{\mathfrak G}

\begin{document}

\date{\today}

\title{Torus bundles not distinguished by TQFT 
invariants}
    
 \author[L.Funar]{Louis Funar\\
\vspace{0.5cm}
\em with an appendix by Louis Funar and Andrei Rapinchuk}     
\address{Institut Fourier BP 74, UMR 5582, 
University of Grenoble I, 
38402 Saint-Martin-d'H\`eres cedex, France}      
\email{funar@fourier.ujf-grenoble.fr}  

\address{
University of Virginia Department of Mathematics
141 Cabell Drive, Kerchof Hall, PO Box 400137,
Charlottesville, VA 22904-4137, USA}
\email{asr3x@virginia.edu}

\begin{abstract}
We show that there exist infinitely many 
pairs of non-homeomorphic closed oriented SOL torus bundles with the same
quantum (TQFT)  invariants. 
This follows from the arithmetic behind the conjugacy problem in
$SL(2,\Z)$ and its congruence quotients, the classification of SOL 
(polycyclic) 3-manifold groups and an elementary study of a 
family of Pell equations. A key ingredient is the congruence 
subgroup property of modular representations, 
as it was established by Coste and Gannon, Bantay, Xu for 
various versions of TQFT, and lastly by Ng and Schauenburg for the 
Drinfeld doubles of spherical fusion categories. 
On the other side we prove that two torus bundles over the circle 
with the same quantum invariants are (strongly) commensurable.
The examples above show that this is the best that it could be expected.

\vspace{0.1cm}
 

\end{abstract}

\maketitle

\section{Introduction and statements}

Two fundamental constructions of TQFTs are due to Reshetikhin-Turaev 
(see \cite{RT}), using link invariants and quantum groups, 
and to Turaev-Viro (\cite{TuV}), using quantum 6j-symbols. 
The Reshetikhin-Turaev method was further extended in \cite{Tu} to a 
a very general construction of TQFTs, whose input is a modular tensor 
category, namely an algebraic structure  
which seems to be the most general data 
needed for building invariants of arbitrary 
closed 3-manifolds.

If $\mathcal A$ is such a modular tensor category (see \cite{Tu}) we 
denote by $RT_{\mathcal A}$ the Reshetikhin-Turaev 
TQFT invariant of 3-manifolds constructed out of the category ${\mathcal A}$. 
In the particular case when the modular tensor category $\mathcal A$ is 
the Drinfeld double  $D(\mathcal C)$ of a spherical fusion category 
(also called the center of $\mathcal C$) the associated invariant $RT_{D(\mathcal C)}$ will be denoted as 
$TV_{\mathcal C}$ and it will be called  
the Turaev-Viro TQFT invariant of 3-manifolds 
associated to $\mathcal C$. If ${\mathcal C}$ were itself a modular tensor 
category then $RT_{D(\mathcal C)}$ would indeed coincide with 
the usual Turaev-Viro invariants $|M|_{\mathcal C}$ constructed out 
of ${\mathcal C}$ by intrinsic methods (see \cite{Tu}, section V). 
We chose to single out this  
family of TQFT invariants because they are somewhat easier 
to handle than the more general Reshetikhin-Turaev invariants, as 
they lead to anomaly-free TQFTs. 
Observe also that spherical fusion categories are more general 
than modular tensor categories although their Drinfeld doubles account  
only for part of the anomaly-free modular tensor categories.

According to M\"uger's results (see \cite{Mu}) 
the Drinfeld double $D(\mathcal C)$ of a spherical fusion 
category $\mathcal C$ is a modular tensor category. 
As a matter of terminology, the Turaev-Viro invariants $TV_{\mathcal C}$ should 
not be confused 
with the Turaev-Viro-Barrett-Westbury invariant $|M|_{\mathcal C}$, 
which extends the intrinsic state-sum definition  of a 3-manifold 
invariant associated to an arbitrary spherical fusion 
category ${\mathcal C}$ (see \cite{BW}). 
Nevertheless this source of confusion is not relevant, as 
Turaev and Virelizier proved recently (see \cite{TV}) that  the  
Turaev-Viro-Barrett-Westbury invariant 
$|M|_{\mathcal C}$  actually coincides with 
$RT_{D(\mathcal C)}(M)$, for any spherical 
fusion category ${\mathcal C}$ of non-zero dimension.  
Notice that, according to (\cite{ENO}, Thm.2.3) all spherical fusion categories 
over $\C$ have non-zero dimension. 
All fusion categories considered here will be $\C$-linear categories, unless 
the opposite is explicitly stated.

A natural question in the area is to what extent the collection 
of all these 3-manifolds invariants determine the topology 
of the manifolds. The aim of this article is to solve this 
question for a particular class of 3-manifolds, namely the 
SOL manifolds. 

Every closed SOL manifold 
has a finite cover of degree at most $8$ which is a torus bundle over 
a circle. Given $A\in SL(2,\Z)$ we denote 
by $M_A$ the torus bundle over the circle 
whose monodromy is given by the matrix $A$. 
It is well-known that the manifold has geometry SOL if and only if $A$ 
is hyperbolic (or Anosov). 

The first result of this paper is the following:

\begin{theorem}\label{TVagree}
There exist infinitely many pairs of Anosov matrices $A$, $B$ such that 
$M_A$ and $M_B$ have non-isomorphic fundamental groups 
although for every spherical fusion category 
${\mathcal C}$ their Turaev-Viro invariants agree: 
\begin{equation} 
TV_{\mathcal C}(M_A)=TV_{\mathcal C}(M_B)
\end{equation} 
The simplest  series of examples is the following:  
\begin{equation}
A=\left(\begin{array}{cc}
1 & kq^2 \\
kv & 1+k^2q^2v \\
\end{array}
\right), \, B=\left(\begin{array}{cc}
1 & k \\
kvq^2 & 1+k^2q^2v \\
\end{array}\right)
\end{equation}
where $k\in\Z$, $k\neq 0$,  
$q$ is an odd prime number  $q\equiv 1({\rm mod}\, 4)$,  
$v$ is a positive integer such that $-v$ is a non-zero quadratic residue 
mod $q$ and $v$ is  divisible either 
by a prime $p$ satisfying $p\equiv 3({\rm mod}\, 4)$, or by $4$.
\end{theorem}

\begin{remark}
Notice that the manifolds $M_A$ and $M_B$ are 
prime SOL manifolds. 
\end{remark}

As an immediate consequence we obtain a negative answer to 
a question due to Turaev (see \cite{Tu}, Problem 5, p.571).

\begin{corollary}\label{RTagreeweak}
There exist infinitely many pairs of matrices $A$ and $B$ as in Theorem \ref{TVagree} 
such that $M_A\#\overline{M_A}$ and  $M_B\#\overline{M_B}$
have non-isomorphic fundamental groups but for every modular tensor 
category  ${\mathcal C}$ their 
Reshetikhin-Turaev TQFT invariants  agree: 
\begin{equation} 
RT_{\mathcal C}(M_A\#\overline{M_A})=RT_{\mathcal C}(M_B\#\overline{M_B})
\end{equation} 
Here $\overline{M}$ denotes the manifold $M$ with the reversed 
orientation. 
\end{corollary}

\begin{proof}
This follows from the fact that 
\begin{equation}
RT_{\mathcal C}(M_A\#\overline{M_A})=TV_{D(\mathcal C)}(M_A)
\end{equation}
according to Proposition \ref{RTV}. 
Moreover prime decomposition of 3-manifolds, as well as splittings 
of groups as free amalgamated products are unique 
by classical results of Milnor and Stallings. 
Therefore the fundamental groups are non-isomorphic since their 
factors are not isomorphic.  
\end{proof}

We will show later (see Theorem \ref{RTagree}) that 
there are also examples of pairs of prime manifolds, 
but we cannot provide yet infinite families.

Recall now that a quotient of $SL(2,\Z)$ is a 
{\em congruence quotient} if it is of the form $SL(2,\Z/m\Z)$ 
for some non-zero integer $m$. 
The key steps in the proof of Theorem \ref{TVagree} 
are the following. We will prove first: 

\begin{proposition}\label{trace}
If $M_A$ and $M_B$ are torus bundle as above then 
$TV_{\mathcal C}(M_A)=TV_{\mathcal C}(M_B)$ for any spherical fusion category 
provided that  the matrices 
$A$ and $B$ are conjugate in every congruence quotient of $SL(2,\Z)$. 
\end{proposition}

\begin{remark}
It seems plausible that  
$RT_{\mathcal A}(M_A)=RT_{\mathcal A}(M_B)$ for every modular tensor category ${\mathcal A}$ whose TQFT is anomaly-free if and only if 
$A$ and $B$ are conjugate in every congruence quotient of $SL(2,\Z)$.
\end{remark}

Lackenby was the first to observe in \cite{L} that 
quantum $SU(2)$-invariants behave well 
with respect to modular transformations from congruence subgroups. 
Specifically he defined the $f$-congruence of manifolds, for $f\in \Z_+\setminus\{0,1\}$,  as follows. Two closed 3-manifolds are 
$f$-congruent if they can be obtained by Dehn surgeries on framed links 
related by a sequence of moves which consists in Kirby moves and changes 
of the framings by adding integral multiples of $f$.  
This was further explored  and refined (to weak and strong $f$-congruence) 
by Gilmer in \cite{Gil} where it was shown that 
quantum invariants are natural obstructions to the $f$-congruence of 
given 3-manifolds. 

We will say that two closed 3-manifolds are {\em congruent} if 
they are $f$-congruent, for {\em every} 
integral $f$.  Therefore the meaning of our Proposition \ref{trace} is that 
the torus bundles $M_A$ and $M_B$ are congruent.

Now there exists an explicit classification of the manifolds of 
the form $M_A$. For the sake of simplicity we will restrict ourselves 
to Anosov matrices $A,B$. In this case $M_A$ is a SOL manifold and it is 
easy to see that it is Haken since the fiber is incompressible. 
Therefore it suffices to understand its  
fundamental group, which is the polycyclic group $\Gamma_A$ with the 
presentation:
\begin{equation} 
\Gamma_A=\langle t,a,b | ab=ba, tat^{-1}=a^{\alpha_{11}}b^{\alpha_{12}}, 
tbt^{-1}=a^{\alpha_{21}}b^{\alpha_{22}}\rangle 
\end{equation} 
where 
$A=\left(\begin{array}{cc}
\alpha_{11} & \alpha_{12} \\
\alpha_{21} & \alpha_{22} \\
\end{array}
\right)$. 
We have then the following: 

\begin{proposition}\label{classif}
Let $A$ and $B$ be matrices from $SL(2,\Z)$ whose traces are different from 
$2$. Then the groups $\Gamma_A$ and $\Gamma_B$ are isomorphic if and only if 
$A$ is conjugate to either $B$ or to $B^{-1}$ within $GL(2,\Z)$. 
\end{proposition}

Although considered a folklore statement  going back as far as Poincar\'e 
the result above  seems to have  first appeared
with a sketch of proof in (\cite{GS}, Appendix 1, Prop.2) 
and then with all details in the unpublished 
\cite{Bar}. For the sake of completeness 
we give a detailed proof below. Notice that Proposition \ref{classif} 
actually gives the  classification of torus bundles up to homeomorphism, 
since these are aspherical Haken manifolds and hence  completely determined 
by their fundamental groups.

Eventually the problem of finding 3-manifolds $M_A$ and $M_B$ as in the 
statement of Theorem \ref{TVagree} is reduced to a purely arithmetic question 
on integral matrices. This amounts to find whether there exist Anosov 
integral matrices which are conjugate in every congruence subgroup 
but are not conjugate within $GL(2,\Z)$. This question was already answered  
affirmatively by Stebe in \cite{S}, who gave such an example. 
We are able to give infinitely many such pairs of examples 
having a slightly stronger property (as needed in 
Proposition \ref{classif}), as follows:

\begin{proposition}\label{noteq}
There exist infinitely many pairs of matrices $A$ and $B$ in $SL(2,\Z)$ 
which are conjugate in 
every congruence quotient, such that $A$ is  conjugate neither to 
$B$ nor to $B^{-1}$ in $GL(2,\Z)$. For instance we can take 
$A=\left(\begin{array}{cc}
1 & kq^2 \\
kv & 1+k^2q^2v \\
\end{array}
\right)$ and $B=\left(\begin{array}{cc}
1 & k \\
kvq^2 & 1+k^2q^2v \\
\end{array}
\right)$, where $k\in\Z$,  
$q$ is an odd prime number $q\equiv 1({\rm mod})\, 4)$,   
$v$ is a positive integer such that first $-v$ is a  non-zero quadratic residue 
mod $q$, and second $v$ is  divisible either 
by a prime $p\equiv 3({\rm mod}\, 4)$, or by $4$. 
\end{proposition}

\begin{remark}
Stebe's example from \cite{S} is 
$A=\left(\begin{array}{cc}
188 & 275 \\
121 & 177 \\
\end{array}
\right)$ and $B=\left(\begin{array}{cc}
188 & 11 \\
3025 & 177 \\
\end{array}
\right)$.  
\end{remark}

This implies that any  pair of  integral Anosov matrices as in Proposition 
\ref{noteq} gives raise to  SOL 3-manifolds which are not distinguished 
by their Turaev-Viro TQFT invariants, thus proving Theorem \ref{TVagree}.

In the examples above the manifolds $M_A$ and $M_B$ 
obtained throughout Proposition \ref{noteq} 
are actually commensurable SOL manifolds. 
This is not a fortuitous coincidence since we have the following: 

\begin{theorem}\label{comm}
If the torus bundles SOL manifolds $M$ and $N$ have the same 
Turaev-Viro invariants for the $U(1)$ and $SU(2)$ TQFTs 
then they are  strongly commensurable.
\end{theorem}
Let us explain briefly the terminology used for the commensurability above. 
Two groups are said to be {\em commensurable} if they have 
finite index subgroups which are isomorphic. 

Barbot (\cite{Bar}, see also \cite{BG}) proved that the groups 
$\Gamma_A$ and $\Gamma_B$ are commensurable if and only if the quotient 
of their discriminants $D_A/D_B$ is the square of a rational number.  
Here the discriminant of $A$ is $D_A={\rm Tr}(A)^2-4\det(A)$, when 
${\rm Tr}(A)$ is odd. Moreover, this is 
equivalent to the fact that $A^p$ and $B^q$ are conjugate within $GL(2,\Q)$, 
for some $p,q\in\Z\setminus\{0\}$. We will call the matrices $A$ and $B$ in $SL(2,\Z)$ 
{\em strongly commensurable} if actually $A$ and $B$ are 
conjugate within $GL(2,\Q)$, 
namely they have the same trace (and determinant).

Let us introduce some more terminology coming from classical 
class field theory.  We set 
${\mathcal I}(M_A)$ for the ideal class group of the 
order $\Z\left[\frac{{\rm Tr}(A)+\sqrt{D_A}}{2}\right]$ 
of the real quadratic field $\Q(\sqrt{D_A})$. When $D_A$ is squarefree 
the order is the ring of integers of  $\Q(\sqrt{D_A})$.  
An old Theorem of Latimer, MacDuffee and Taussky-Todd 
(see \cite{TT} and \cite{New}, III.16) 
shows that there is a one-to-one correspondence 
between ${\mathcal I}(M_A)$ and the set of matrices 
$B$ from $SL(2,\Z)$ having the same 
trace as $A$, which are considered up to conjugacy in $GL(2,\Z)$. 
In this context the ``taking the inverse'' map 
$B\to B^{-1}$ passes to the quotient 
and gives a well-defined involution  
$\iota:{\mathcal I}(M_A)\to {\mathcal I}(M_A)$.  

Let $M$ be a given closed orientable 3-manifold. Denote by 
${\mathcal X}^{U(1),SU(2)}(M)$  (and  
${\mathcal X}^{TV}(M)$) the set of homeomorphisms classes 
of closed orientable 3-manifolds $N$ having the same 
abelian, $SU(2)$ Turaev-Viro invariants (and 
the same Turaev-Viro invariants, for every spherical fusion category, 
respectively).

\begin{corollary}\label{class}
Let $M$ be a SOL torus bundle over the circle. The subset of the  
torus bundles homeomorphism classes in   
${\mathcal X}^{U(1),SU(2)}(M)$ injects into ${\mathcal I}(M)/\iota$ and, in particular, it is bounded by the class number of the corresponding 
totally real quadratic field.
\end{corollary}

\begin{remark}
One might consider the set of torus bundles 
$N$ having the same Turaev-Viro invariants 
as $M$, up to an {\em orientation preserving homeomorphism}. 
Observe that $M_A$ and $M_B$ are orientation-preserving 
homeomorphic if and only if $A$ and $B$ are conjugate within $SL(2,\Z)$ 
or else $A$ and $B^{-1}$ are conjugate in $GL(2,\Z)$ by a matrix 
of determinant $-1$. 
\end{remark}

The few examples  we know would suggest that  the subset 
of torus bundles homeomorphism classes in ${\mathcal X}^{TV}(M)$  
is a quite small proper subset of ${\mathcal I}(M)/\iota$. 
As a consequence of our proof of Theorem \ref{TVagree} and 
results of Platonov and Rapinchuk (see \cite{Platonov,Rapin}, 
\cite[ section 8.8.5]{PR}) 
on the genus problem in arithmetic groups we obtain 
a stronger but less precise statement as follows: 

\begin{corollary}\label{plat}
The number of homeomorphisms classes of torus bundles in 
${\mathcal X}^{TV}(M)$, for  $M$ running 
over all torus bundles, is unbounded. Alternatively, for each $m \geq 2$ 
there exist examples of $m$ pairwise non-homeomorphic torus bundles 
having the same Turaev-Viro invariants for all spherical fusion categories.  
\end{corollary}

We can slightly improve the finiteness result in Corollary \ref{class}, 
as follows: 

\begin{proposition}\label{finiteclass}
If $M$ is a closed irreducible orientable SOL manifold then 
$|{\mathcal X}^{TV}(M)|$ is finite.  
\end{proposition}

These results give some evidence for the following general conjecture: 

\begin{conjecture}
If $M$ and $N$ are closed irreducible geometric 3-manifolds 
having  the same abelian and $SU(2)$ Turaev-Viro invariants 
then $M$ and $N$ should be commensurable and, in particular, they share 
the same geometry. 
\end{conjecture}

On the other hand, we don't know whether the unboundedness 
of the number of classes of torus bundles 
is a general phenomenon, valid in higher genus as well. 
In order to dismiss obvious examples constructed out of torus bundles 
we ask:  
\begin{conjecture}
The number of homeomorphism classes in 
${\mathcal X}^{TV}(M)$ of hyperbolic 
fibered 3-manifolds $N$  with fiber of genus $g\geq 1$ is finite 
for every $M$.  Is this number unbounded, 
when $M$ runs over the set of hyperbolic fibered 3-manifolds 
with fiber of given genus? 
\end{conjecture}
 
The pairs of manifolds from Theorem \ref{TVagree} and Corollary \ref{plat} 
also give a negative answer to a question stated by Long and Reid 
in \cite{LR}  (see also Remark 3.7 in \cite{CFW}), as follows: 

\begin{corollary}\label{profinite}
For any $m\geq 2$ there exist torus bundles whose 
fundamental groups have isomorphic profinite completions
although they are pairwise not isomorphic. 
\end{corollary}

This consequence was independently noticed by G.Masbaum. 

We don't know if the $SU(2)$ Turaev-Viro invariants alone determine 
already the profinite completion of the fundamental group. 

\begin{remark}
We expect that the topological content of the Turaev-Viro invariants 
is precisely this kind of arithmetic information. 
An over-optimistic conjecture would be that two  
closed irreducible geometric 3-manifolds $M$ and $N$ with infinite 
fundamental groups define the same class in ${\mathcal X}^{TV}(M)$ 
if and only if  the profinite completions of their 
fundamental groups are isomorphic. The "only if" part is immediate 
(see the proof of Corollary \ref{profinite}).  In particular, if the 
closed hyperbolic 3-manifolds $M$ is determined up to homeomorphism 
by the profinite completion of its fundamental group, then 
${\mathcal X}^{TV}(M)$ will be a singleton. 
This would connect the quantum invariants to some 
version of Grothendieck's problem for 3-manifold groups  
which is stated in \cite{LR}.   
\end{remark}

Eventually, the result of Theorem \ref{TVagree} 
can be improved (with a loss of effectivity),  as follows: 

\begin{theorem}\label{RTagree}
There exist infinitely many pairs of Anosov matrices $A$, $B$ such that 
$M_A$ and $M_B$ have non-isomorphic fundamental groups 
although for every modular tensor category 
${\mathcal C}$ their Reshetikhin-Turaev invariants agree: 
\begin{equation} 
RT_{\mathcal C}(M_A)=RT_{\mathcal C}(M_B)
\end{equation} 
The simplest four examples are the following:  
\begin{equation}
A=
\left(\begin{array}{cc}
1 & 21 \\
21 & 442 \\
\end{array}
\right), \, B=\left(\begin{array}{cc}
106 & 189 \\
189 & 337  \\
\end{array}\right),
\end{equation}
\begin{equation}
A=
\left(\begin{array}{cc}
1 & 51 \\
51 & 2602 \\
\end{array}
\right), \, B=\left(\begin{array}{cc}
562 & 1479 \\
1479 & 2041 \\
\end{array}\right),
\end{equation}
\begin{equation}
A=
\left(\begin{array}{cc}
1 & 53 \\
53 & 2810 \\
\end{array}
\right), \, B=\left(\begin{array}{cc}
425 & 1007 \\
1007 & 2386 \\
\end{array}\right),
\end{equation}
\begin{equation}
A=
\left(\begin{array}{cc}
1 & 55 \\
55 & 3026 \\
\end{array}
\right), \, B=\left(\begin{array}{cc}
881 & 1375 \\
1375 & 2146 \\
\end{array}\right).
\end{equation}
\end{theorem}

The equivalence relation on torus bundles induced by 
the equality of all Turaev-Viro invariants is the local equivalence 
of matrices determining a fixed {\em genus}, in the sense 
studied by Platonov  and Rapinchuk (see \cite{Platonov,Rapin}, 
\cite[ section 8.8.5]{PR}). 
Specifically, $M_B$ and $M_A$ represent the same 
class in ${\mathcal X}^{TV}(M)$  if and only if 
$A$ and $B$ are locally conjugate, namely their images  mod $m$ are 
conjugate in  
$GL(2,\Z/m\Z)$, for any positive integer $m$. Notice that this 
implies automatically that $A$ and $B$ are conjugate in  $GL(2,\Q)$.
     
A related equivalence relation is the one corresponding to the 
{\em Pickel genus} of groups (see \cite{Pickel}). 
Two finitely generated groups are in the same 
Pickel genus if the corresponding sets of finite quotients are the same. 
This is equivalent, following a deep result 
of Nikolov and Segal (see \cite{NiS}) to the fact that 
their profinite completions are isomorphic.  
The groups of torus bundles $\pi_1(M_B)$ and $\pi_1(M_A)$ 
have isomorphic profinite completions if and only if 
the subgroups $\langle A \rangle$ and $\langle B \rangle$ 
are locally conjugate, namely their images mod $m$ are conjugate in  
$GL(2,\Z/m\Z)$, for any positive integer $m$. 
This is coarser than the former equivalence relation.

\vspace{0.2cm}
{\bf Acknowledgements.} We are indebted to Christian Blanchet, Thierry Barbot,  
Francois Costantino, Michael Freedman, 
Terry Gannon, Jurgen Klueners, Greg Kuperberg, Gregor Masbaum, 
Greg McShane, Alan Reid, Chris Schommer-Pries, Vlad Sergiescu, Peter 
Stevenhagen, Vladimir Turaev, Zhenghan Wang and Maxime Wolff for 
useful discussions and comments.

\section{Preliminaries about modular tensor categories}

\subsection{Fusion  categories} 
For simplicity we will consider only {\em strict} monoidal categories below, 
meaning that the associativity morphisms are identities.  We follow 
the definitions from \cite{BK,Mu}. 

A {\em left/right rigid} monoidal category is a 
a strict
monoidal category ${\mathcal C}$ 
with unit object ${\mathbf 1}$ such that to each
object $X\in \mathrm{Ob}({\mathcal C})$ there are associated a dual
object $X^*\in \mathrm{Ob}({\mathcal C})$ and four morphisms
\begin{eqnarray*}
& \mathrm{ev}_X \colon X^*\otimes X \to{\mathbf 1},  \qquad \mathrm{coev}_X\colon {\mathbf 1}  \to X \otimes X^*,\\
&   \widetilde{\mathrm{ev}}_X \colon X\otimes X^* \to{\mathbf 1}, \qquad   
\widetilde{\mathrm{coev}}_X\colon {\mathbf 1}  \to X^* \otimes X \nonumber 
\end{eqnarray*}
such that, for every $X\in
\mathrm{Ob}({\mathcal C})$, the pair $(\mathrm{ev}_X,\mathrm{coev}_X)$ is a 
left duality for $X$ and the pair $(\widetilde{\mathrm{ev}}_X,\widetilde{\mathrm{coev}}_X)$ is a right
duality for $X$, namely: 
\begin{equation*}
({\mathbf 1}_X \otimes \mathrm{ev}_X)(\mathrm{coev}_X \otimes {\mathbf 1}_X)={\mathbf 1}_X  \quad  {\rm {and}} \quad (\mathrm{ev}_X \otimes {\mathbf 1}_{X^*})({\mathbf 1}_{X^*} \otimes \mathrm{coev}_X)={\mathbf 1}_{X^*}
\end{equation*}
The category is rigid if it is both left and right rigid. 

A {\em pivotal} category is a left rigid monoidal category equipped with 
an isomorphism $j$ of monoidal functors between 
identity and $(-)^{**}$, called pivotal structure. 
One should notice that the formulas  
\begin{equation*}
\widetilde{\mathrm{ev}}_X =
\mathrm{coev}_X(\mathbf 1_{X^*}\otimes j_X^{-1}), \: 
\widetilde{\mathrm{coev}}_X= 
(j_X\otimes \mathbf 1_{X^*})\mathrm{ev}_X 
\end{equation*}
define a right duality so that a pivotal category is rigid. 

It is known that every pivotal category is equivalent to a strict pivotal 
category, namely one where the associativity isomorphisms, the pivotal 
structure and the canonical isomorphisms 
$(V\otimes W)^*\to W^*\otimes V^*$ are identities. 

The morphisms $\mathrm{ev}_{\mathbf 1}$ and  $\mathrm{coev}_{\mathbf 1}$ (respectively,
$\widetilde{\mathrm{ev}}_{\mathbf 1}$ and  $\widetilde{\mathrm{coev}}_{\mathbf 1}$) are mutually inverse isomorphisms and 
$\mathrm{ev}_{\mathbf 1}=\widetilde{\mathrm{ev}}_{\mathbf 1} \colon {\mathbf 1}^* \to {\mathbf 1}$. 

Now, for an endomorphism $f $ of an object $X$ of   
a pivotal category ${\mathcal C}$, one defines the
{\em left/right traces} ${\rm tr}_l(f), {\rm tr}_r(f) \in
\mathrm{End}_{\mathcal C}({\mathbf 1})$ by
\begin{equation*}
{\rm tr}_l(f)=\mathrm{ev}_X({\mathbf 1}_{\ldual{X}} \otimes f) \widetilde{\mathrm{coev}_X  \quad {\text
{and}}\quad
 {\rm tr}_r(f)=\widetilde{\mathrm{ev}}_X( f \otimes {\mathbf 1}_{\ldual{X}}) \mathrm{coev}}_X  
\end{equation*}
Both traces are symmetric: ${\rm tr}_l(gh)={\rm tr}_l(hg)$ and $
{\rm tr}_r(gh)={\rm tr}_r(hg)$ for any morphisms $g\colon X \to Y$ and   $h\colon Y
\to X$  in ${\mathcal C}$. Also ${\rm tr}_l(f)={\rm tr}_r(\ldual{f})={\rm tr}_l(f^{**})$ for
any endomorphism $f $ of an object (and similarly for $l$ exchanged with $r$). 

The left and right dimensions of   $X\in \mathrm{Ob} ({\mathcal C})$
are defined by $ \dim_l(X)={\rm tr}_l({\mathbf 1}_X) $ and $
\dim_r(X)={\rm tr}_r({\mathbf 1}_X) $.  
Note that isomorphic objects have the same
dimensions and $\dim_l({\mathbf 1})=\dim_r({\mathbf 1})={\mathbf 1}_{{\mathbf 1}}$.

A {\em spherical category} is a pivotal category whose left and
right traces are equal, i.e.,  ${\rm tr}_l(f)={\rm tr}_r(f)$ for every
endomorphism $f$ of an object. Then they are 
denoted  ${\rm tr}(f)$ and called the trace of $f$. The
left (and right) dimensions of an object $X$ are denoted  $\dim(X)$
and called the dimension of $X$. In a 
(strict) spherical category we can make free use of the 
graphical calculus. 

Let $\mathbb K$ be a field, which for the moment is not supposed 
to be of characteristic zero, although in the next section we will consider 
$\mathbb K=\C$. 

A monoidal {\em ${\mathbb K}$-linear category} is a monoidal category 
${\mathcal C}$ such that its Hom-sets are (left) ${\mathbb K}$-modules and
the composition and monoidal product of morphisms are ${\mathbb K}$-bilinear. 
An  object $V\in \rm{Ob}(\mathcal C)$ is called {\em simple} 
if the map ${\mathbb K} \to \mathrm{End}_{\mathcal C}({\mathbf 1}), 
a \mapsto a \, {\mathbf 1}_{\mathbf 1}$  is a
${\mathbb K}$-algebra isomorphism. 

An additive category is said to be {\em semisimple} 
if every object is a direct sum 
of finitely many simple objects. In the case of Ab-categories from \cite{Tu} 
we can weaken our requirements by asking that every object 
be dominated by finitely many simple objects.    
A monoidal  ${\mathbb K}$-linear category is called {\em semisimple} if the 
underlying $\mathbb K$-linear category is semisimple with 
finite dimensional Hom spaces and $\mathbf 1$ is a 
simple object. 

Now a {\em fusion} category over $\mathbb K$ is a rigid 
semisimple $\mathbb K$-linear category ${\mathcal C}$ with 
finitely many simple objects. The fusion categories 
which are considered in the next sections will always be 
spherical. 

A monoidal category ${\mathcal C}$ is {\em braided} if there exist 
natural isomorphisms $c_{V,W}:V\otimes W\to W\otimes V$ for every objects 
$V,W$,  such that for any $U,V,W\in {\rm Ob}(\mathcal C)$ we have: 
\begin{equation*}
c_{U,V\otimes W}=(\mathbf 1_{V}\otimes c_{U,W})(c_{U,V}\otimes 
\mathbf 1_{W}), \: 
c_{U\otimes V,W}=(c_{U,W}\otimes 
\mathbf 1_{V})(\mathbf 1_{U}\otimes c_{V,W})
\end{equation*}

Let now ${\mathcal C}$ be a left rigid braided monoidal category. 
We do not require that $V^{**}=V$. 
A twist of ${\mathcal C}$ is an automorphism $\theta$ of the identity 
functor of ${\mathcal C}$ satisfying 
\begin{equation*}
\theta_{V\otimes W}=c_{W,V}c_{V,W}(\theta_V\otimes\theta_W), \: \rm{ and } 
\:\: \theta_{\mathbf 1}={\mathbf 1}_{\mathbf 1}
\end{equation*}
The twist $\theta$ is a {\em ribbon} structure on $(\mathcal C, c)$ 
if it also satisfies $\theta_V^*=\theta_{V^*}$ for every 
$V\in \rm{Ob}(\mathcal C)$, and 
the (left) duality is compatible with the ribbon and twist structures, namely:
\begin{equation*}
(\theta_V\otimes \mathbf 1_{V^*}){\rm coev}_V=
(\mathbf 1_V\otimes \theta_{V^*}){\rm coev}_V
 \end{equation*}
In this case $(\mathcal C, c, \theta)$ is called a {\em ribbon} category. 
In a ribbon category one associates naturally a pivotal structure 
by using the (canonical) isomorphism $u_X:X\to X^{**}$ given by:
\begin{equation*}
u_X=({\rm ev}_{X^*}\otimes {\mathbf 1}_X)({\mathbf 1}_{X^*}\otimes c^{-1}_{X,X^{**}})({\rm coev}_X\otimes {\mathbf 1}_{X^{**}})
\end{equation*}
and setting $\theta=u^{-1}j$. 
Moreover this pivotal structure $j$ is spherical. 

A {\em modular tensor category} over $\mathbb K$ is a ribbon fusion category 
$({\mathcal A},c,\theta)$ over $\mathbb K$ such that the matrix $S$
having entries $S_{ij}=tr((c_{U_j,U_i^*}c_{U_i^*,U_j})$ is non-singular, where 
$i,j\in I$ and $I$ is the set indexing the simple objects 
$U_i, i\in I$ in ${\mathcal A}$. This  matrix 
is called the $S$-matrix of the category ${\mathcal A}$. 
Notice that $I$ has induced 
a duality $*$ such that $U_{i^*}=U_i^*$, for any $i\in I$ and 
there exists a label (also called color) $0\in I$ such that 
$U_0=\mathbf 1$.  
Since the object $U_i$ is simple the twist $\theta_{U_i}$ acts on
$U_i$ as a scalar $\omega_i\in \mathbb K$.  

The (left) {\em Drinfeld double} (also called the center) 
of a (strict) monoidal category ${\mathcal C}$ is a category $D(\mathcal C)$ 
whose objects are pairs $(V, \sigma_V)$, where $V\in \rm{Ob}(\mathcal C)$ 
and  the half-braiding $\sigma_V(W):V\otimes W\times W\otimes V$ are natural 
isomorphisms satisfying for every $U,V,W\in \rm{Ob}(\mathcal C)$ the 
identities:   
\begin{equation*}
(V\otimes \sigma_U(W))(\sigma_{U}(V)\otimes W)=
\sigma_U(V\otimes W), \: \sigma_V(\mathbf 1)=\mathbf 1_X
\end{equation*}
There is a natural monoidal structure on $D(\mathcal C)$ by defining 
the tensor product 
$(U, \sigma_U)\otimes (V,\sigma_V)=
(U\otimes V, \sigma_{U\otimes V})$, where 
\begin{equation*}
\sigma_{U\otimes V}(W)=(\sigma_U(W)\otimes V)(U\otimes \sigma_V(W))
\end{equation*}
and the unit object $(\mathbf 1, \sigma_{\mathbf 1})$, where 
$\sigma_{\mathbf 1}(V)={\mathbf 1}_V$, for any $U,V,W\in \rm{Ob}(\mathcal C)$. 
More interesting is the fact that $D(\mathcal C)$ has a braiding given by 
$c_{(V, \sigma_V), (W, \sigma_W)}=\sigma_{V}(W)$ so that 
${\mathcal C}$ is a braided monoidal category. 
If ${\mathcal C}$ is left rigid/pivotal/spherical 
then  $D(\mathcal C)$ is also left rigid/pivotal/spherical respectively.

\subsection{$SL(2,\Z)$ representations  from modular tensor categories}
Any modular tensor category ${\mathcal C}$ defined over the algebraically closed  field $\mathbb K$ has associated 
the modular data (see \cite{G}), which contains a {\em projective} 
representation $\overline{\rho}_{\mathcal C}:SL(2,\Z)\to PGL(\mathcal K_0(
\mathcal C))$, where $\mathcal K_0(
\mathcal C))$ is the Grothendieck ring of $\mathcal C$ with $\C$-coefficients. 
However, we have slightly more than that, namely a lift of  
$\overline{\rho}_{\mathcal C}$ to an almost linear representation, 
by means of the matrices $S$ and $T$. The almost linear representation 
comes with a 2-cocycle which was completely described by Turaev. 
An essential feature of the genus 1 situation is that projective 
representations could always be lifted (in more than one way) 
to genuine linear representations, which contrasts with the higher genus case. 

The matrices entering in the definition of $\overline{\rho}_{\mathcal C}$ 
are the $S$-matrix  defined above and  the $T$-matrix associated to the twist. 
Specifically, $T$ has the entries $T_{ij}=\omega_i\delta_{ij}$, $i,j\in I$.  
Moreover there is also the so-called charge conjugation 
matrix $C$ having entries  $C_{ij}=\delta_{i\,j^*}$, $i,j\in I$, which is actually $S^2$.
 
The Gauss sums of ${\mathcal C}$ are given by 
$p_{\mathcal C}^{\pm}=\sum_{i\in I}\omega_i^{\pm 1}\dim(U_i)^2$ and these are non-zero scalars satisfying: 
\begin{equation}
p_{\mathcal C}^{+}p_{\mathcal C}^{-}=\sum_{i\in I} \dim(U_i)^2=\dim(\mathcal C)
\end{equation}

In \cite{Tu} Turaev used the notation 
$\Delta_{\mathcal C}=p_{\mathcal C}^{-}$ so that  
$p_{\mathcal C}^{+}= \dim(\mathcal C)\Delta_{\mathcal C}^{-1}$. Further 
one chooses a {\em rank} (also called quantum order), which is an element 
${\lambda}\in \mathbb K$ such that 
${\lambda}^2=\dim {\mathcal C}$. This was denoted by ${\mathcal D}$ in 
\cite{Tu}.  

The group $SL(2,\Z)$ is generated by the matrices 
$\mathfrak s=\left(\begin{array}{cc}
0 & -1 \\
1 & 0 \\
\end{array}
\right)$ and 
$\mathfrak t=\left(\begin{array}{cc}
1 & 1 \\
0 & 1 \\
\end{array}
\right)$. The usual presentation of $SL(2,\Z)$ in the generators 
 $\mathfrak s, \mathfrak t$ has the relations 
$(\mathfrak s\mathfrak t)^3=\mathfrak s^2$ and $\mathfrak s^4=1$. 

The projective representation 
 $\overline{\rho}_{\mathcal C}:SL(2,\Z)\to PGL(\mathcal K_{0}(C))$
is defined by 
\begin{equation}
\overline{\rho}_{\mathcal C}(\mathfrak s)=S, \,\,
\overline{\rho}_{\mathcal C}(\mathfrak t)=T
\end{equation}

However the choice of a rank ${\lambda}$ and a third root of unity 
$\zeta\in \mathbb K$ of the anomaly 
$\zeta^3={\mathcal D}\Delta^{-1}=p_{\mathcal C}^{+}{\lambda}^{-1}$  enables us to define a lift of $\overline{\rho}_{\mathcal C}$ 
to an ordinary linear representation
$\rho_{\mathcal C}^{\lambda, \zeta}:SL(2,\Z)\to GL(\mathcal K_{0}(C))$ 
by setting 
\begin{equation}
\rho_{\mathcal C}^{\lambda, \zeta}(\mathfrak s)=\lambda^{-1}S, \,\,
\rho_{\mathcal C}^{\lambda, \zeta}(\mathfrak t)=\zeta^{-1}T
\end{equation}

These lifts are called the modular representations associated to 
$\mathcal C$. It is known that, given a rank $\lambda$ then the 
modular tensor category defines a TQFT with anomaly in the group generated 
by $\zeta^3$, so that 3-manifold invariants associated to the data 
$(\mathcal C,\lambda)$ do not depend on the particular choice of $\zeta$.

\section{Proof of Proposition \ref{trace}}

\subsection{TQFT coming from centers of spherical fusion categories}
In the case when the modular tensor category 
is the Drinfeld double $D({\mathcal C})$
of a spherical fusion category $\mathcal C$ a number of simplifications occur. 

For every $SL(2,\Z)$ representation $\rho$ we define its dual 
representation $\widetilde{\rho}$ by means of 
$\widetilde{\rho}(x)=\rho(JxJ^{-1})$, where $J=\left(\begin{array}{cc}
-1 & 0 \\
0 & 1 \\
\end{array}
\right)\in GL(2,\Z)$ acts by conjugacy as an outer automorphism 
of $SL(2,\Z)$. 

Notice that in this case we have the following: 

\begin{lemma}\label{double}
\begin{enumerate} 
\item The anomaly of the TQFT coming from $D({\mathcal C})$ is trivial, i.e. 
$\zeta^3=1$ and thus there exists a privileged modular representation 
$\rho_{\mathcal C}^{\lambda, 1}$. 
\item Further we have 
$\rho_{D(\mathcal C)}^{\lambda,1}=\rho_{\mathcal C}^{\lambda,\zeta}\otimes 
\widetilde{\rho^{\lambda,\zeta}}_{\mathcal C}$. Here $\zeta$ is arbitrary and in fact 
the right hand side tensor product is well-defined even when 
we have only projective representations.  
\end{enumerate}
\end{lemma}
\begin{proof}
See (\cite{NS}, Lemma 6.2). 
\end{proof}

The invariants of mapping tori have a very simple expression when 
the TQFT is anomaly-free. In fact we have the following well-known result: 
\begin{lemma}\label{anomalyfree}
Assume that  the TQFT associated to  the modular tensor 
category ${\mathcal C}$ is anomaly-free, namely 
that $\zeta^3=1$. Then the invariant of the mapping torus $M_{A}$ 
of $A\in SL(2,\Z)$ is expressed as: 
\begin{equation}
RT_{\mathcal C}(M_{A})={\rm Tr}(\rho_{\mathcal C}^{\lambda,1}(A)) 
\end{equation}
\end{lemma}
\begin{proof}
For the sake of completeness here is the proof. Turaev defined in 
(\cite{Tu}, section IV.5, (5.1.a)) an almost 
linear representation $\epsilon:SL(2,\Z)\to GL({\mathcal K}_{\mathbb K}({\mathcal C}))$ satisfying the cocycle law: 
\begin{equation}
\epsilon(A_1A_2)=\zeta^{3\mu({A_2}_*(L), L, {A_1}_*^{-1}(L))}\epsilon(A_1)\epsilon(A_2)
\end{equation}
where $\mu(L_1,L_2,L_3)$ denotes the Maslov index (see \cite{Tu}, section 
IV.3, p.179) of the triple $(L_1,L_2,L_3)$ of 
Lagrangian subspaces of $H_1(\Sigma_g;\R)$ and $L$ a fixed Lagrangian subspace. 
Further $\epsilon$ is determined by  its values on the generators 
$\epsilon({\mathfrak s})=\lambda^{-1}S$ and 
$\epsilon({\mathfrak t})=T$. If $\zeta^3=1$ then $\epsilon$ is a 
linear representation which coincides with $\rho_{\mathcal C}^{\lambda,1}$. 
Moreover, one also knows from (\cite{Tu}, section III.2.8, Ex.1) that 
\[RT_{\mathcal C}(M_{A})={\rm Tr}(\epsilon(A))\]
This proves the claim.   
\end{proof}

We will prove now: 
\begin{proposition}\label{modularaf}
Let $\mathcal C$ be an anomaly-free modular tensor category such that  
the  modular representation 
$\rho_{\mathcal C}^{\lambda, 1}$ factors through $SL(2,\Z/N\Z)$.
Let $A$ and $B$ be two integral matrices from $SL(2,\Z)$ 
whose reductions mod $N$ are conjugate. 
Then $RT_{\mathcal C}(M_A)=RT_{\mathcal C}(M_B)$.  
\end{proposition}
\begin{proof}
According to the Lemma \ref{anomalyfree} the invariant 
$RT_{\mathcal C}(M_A)$ is the trace of the endomorphism 
$\rho_{\mathcal C}^{\lambda, \zeta}(A)$. By hypothesis 
$\rho_{\mathcal C}^{\lambda, \zeta}$ factorizes as 
 $\rho_{\mathcal C}^{\lambda, \zeta}=
\rho_{\mathcal C}^{\lambda, \zeta,N}\circ \nu_N $, where 
$\nu_N:SL(2,\Z)\to SL(2,\Z/N\Z)$ is  the homomorphism of 
reduction mod $N$.

Since $\nu_N$ is surjective there exists 
$T\in SL(2,\Z)$ such that $\nu_N(A)=\nu_N(T^{-1}BT)$. 
Therefore 
\begin{equation}
{\rm Tr}(\rho_{\mathcal C}^{\lambda, \zeta}(A))=
{\rm Tr}\left(\rho_{\mathcal C}^{\lambda, \zeta,N}(\nu_N(T))\cdot 
\rho_{\mathcal C}^{\lambda, \zeta,N}(\nu_N(B))\cdot 
\left(\rho_{\mathcal C}^{\lambda, \zeta,N}(\nu_N(T))\right)^{-1}\right)=
{\rm Tr}(\rho_{\mathcal C}^{\lambda, \zeta}(B)) 
\end{equation}
Then Lemma \ref{anomalyfree} yields the equality of quantum invariants of 
$M_A$ and $M_B$. 
\end{proof}

The final ingredient in the proof of Proposition \ref{trace} 
is the following result due to 
Ng and Schauenburg for modular tensor categories 
which are centers of spherical fusion categories (\cite{NS}),  
to Peng Xu for conformal field theories derived from 
vertex operator algebras (see \cite{Xu}) and 
to Coste, Gannon and Bantay for RCFT (see \cite{CG,B}). 
Recall that a {\em congruence} subgroup of $SL(2,\Z)$ is 
the kernel of one of the reducing mod $m$  
homomorphism $SL(2,\Z)\to SL(2,\Z/m\Z)$, for some non-zero integer $m$. 

\begin{theorem}\label{congru}
Let $D(\mathcal C)$ be the Drinfeld double of a spherical fusion category 
${\mathcal C}$. Then the modular representations 
$\rho_{D(\mathcal C)}^{\lambda, 1}$ have the congruence property, namely 
the kernels contain congruence subgroups. 
\end{theorem} 

This proves Proposition \ref{trace}, namely 
if $A$ and $B$ are conjugate in {\em every} congruence quotient 
of $SL(2,\Z)$ then  $TV_{\mathcal C}(M_A)=TV_{\mathcal C}(M_B)$ for any 
spherical tensor category $\mathcal C$.

\begin{corollary}
Assume that the modular representation of every anomaly-free 
modular tensor category has the congruence property. 
Then there exist pairs of matrices $A$, $B$ such that 
$M_A$ and $M_B$ are not homeomorphic but 
$RT_{\mathcal A}(M_A)=RT_{\mathcal A}(M_B)$ for any anomaly-free 
modular tensor category ${\mathcal A}$. 
\end{corollary}

\subsection{General  modular tensor categories}
Turaev constructed in (\cite{Tu}, p.198-199) some almost linear 
representations of the mapping class group 
${\mathfrak M}_g$ of genus $g$ surfaces, for every $g$. 
We have to choose first some Lagrangian subspace 
$L\subset H_1(\Sigma_g;\R)$ with respect to the 
usual symplectic form $\omega$ in homology coming from the intersection form.
We denote by $Z_{\mathcal C}(\Sigma_g)$ the space of conformal blocks in genus 
$g$ associated to the modular tensor category ${\mathcal C}$.

It is known that there exist maps (which will be called almost 
linear representations) 
$f_{\mathcal C}^{\lambda,L}: {\mathfrak M}_g\to GL(Z_{\mathcal C}(\Sigma_g))$ 
into the automorphisms of the space of conformal blocks 
$Z_{\mathcal C}(\Sigma_g)$ satisfying the following 2-cocycle condition: 
\begin{equation}\label{cocycle}
f_{\mathcal C}^{\lambda,L}(\varphi_1\varphi_2)=\zeta^{3\mu({\varphi_2}_*(L), L, {\varphi_1}_*^{-1}(L))}f_{\mathcal C}^{\lambda,L}(\varphi_1)f_{\mathcal C}^{\lambda,L}(\varphi_2)
\end{equation}
where $\mu(L_1,L_2,L_3)$ denotes the Maslov index (see \cite{Tu}, section 
IV.3, p.179) of the triple $(L_1,L_2,L_3)$ of 
Lagrangian subspaces of $H_1(\Sigma_g;\R)$.  
This can be found for instance either in (\cite{Tu}, section IV.5, (5.1.a)) 
and also in  an rather equivalent context in 
(\cite{Tu}, section IV.6, Lemma 6.3.2, (6.3.c)).

We introduce now the Rademacher Phi function (see \cite{RG})
$\phi_R:SL(2,\Z)\to \Z$ defined as follows. 
\begin{equation}
\Phi_R\left(\begin{array}{cc}
\alpha & \beta \\
\gamma & \delta \\
\end{array}
\right)=\left\{\begin{array}{ll}
\frac {\alpha+\delta}{\gamma}-12{\rm sgn}(\gamma) s(\alpha,|\gamma|), & {\rm if }\, \gamma\neq 0 \\
\frac {\beta}{\gamma}, &{\rm otherwise}
\end{array}\right.
\end{equation}
Here $s(m,n)$, for $n>0$, denotes the Dedekind sum
\begin{equation}
s(m,n)=\sum_{j=1}^{n-1}=\left(\left(\frac{j}{n}\right)\right)
\left(\left(\frac{jm}{n}\right)\right), \,\, s(0,1)=0
\end{equation}
where 
\begin{equation}
((x))=\left\{\begin{array}{ll}
0, & {\rm if }\, x\in\Z \\
x-[x]-\frac{1}{2} &{\rm otherwise}
\end{array}\right.
\end{equation}
Alternatively we have 
\begin{equation}
s(m,n)=\frac{1}{4n}\sum_{j=1}^{n-1}\cot\left(\frac{\pi j}{n}\right)
\cot\left(\frac{\pi jm}{n}\right), \,\, 
\end{equation}

We have then the following result, which seems to be 
well-known to the specialists: 

\begin{lemma}
Let $L_0$ be the integral Lagrangian subspace of the homology 
$H_1(\Sigma_1;\R)=\R^2$ generated by the vector $(1,0)$.  
Then Turaev's almost linear representation 
$f_{\mathcal C}^{\lambda,L_0}$  in genus $g=1$ is related to 
the modular representation 
$\rho_{\mathcal C}^{\lambda, \zeta}$ of $SL(2,\Z)$, by means of the formula 
\begin{equation}
\rho_{\mathcal C}^{\lambda, \zeta}(A)=\zeta^{-\Phi_{R}(A)}f_{\mathcal C}^{\lambda,L_0}(A)
\end{equation}
for every $A\in SL(2,\Z)$. 
\end{lemma}
\begin{proof}
Consider $A=\left(\begin{array}{cc}
\alpha & \beta \\
\gamma & \delta \\
\end{array}
\right)$ and $B=\left(\begin{array}{cc}
\alpha' & \beta' \\
\gamma' & \delta' \\
\end{array}
\right)$  and let $BA==\left(\begin{array}{cc}
\alpha'' & \beta'' \\
\gamma'' & \delta'' \\
\end{array}
\right)$. By direct computation we obtain 
\begin{equation}\label{maslov}
\mu(BA(L_0), B(L_0),L_0)=-{\rm sgn}(\gamma\gamma'\gamma'')
\end{equation}
 On the other hand Rademacher proved that $\Phi_R$ is a 
a 1-cocycle whose boundary is 3 times the signature 2-cocycle, in 
other words we have the identities: 
\begin{equation}
\Phi_R(BA)=\Phi_R(A)+\Phi_R(B)-3{\rm sgn}(\gamma\gamma'\gamma'')
\end{equation}
for $A,B$ as above. 
Therefore the equation above, the cocycle identity (\ref{cocycle})  
for $f_{\mathcal C}^{\lambda,L_0}$ and (\ref{maslov}) yield:
\begin{equation}
\zeta^{-\Phi_R(BA)}f_{\mathcal C}^{\lambda,L_0}(BA)=
\zeta^{-\Phi_R(B)}f_{\mathcal C}^{\lambda,L_0}(B)\cdot 
\zeta^{-\Phi_R(A)}f_{\mathcal C}^{\lambda,L_0}(A)
\end{equation}
This means that $\zeta^{-\Phi_{R}}f_{\mathcal C}^{\lambda,L_0}$ is a linear 
representation of $SL(2,\Z)$.
Since $\Phi_{R}(\mathfrak s)=0$ and $\Phi_{R}(\mathfrak t)=1$ 
the two linear representations 
$\zeta^{-\Phi_{R}}f_{\mathcal C}^{\lambda,L_0}$  and $\rho_{\mathcal C}^{\lambda, \zeta}$ agree. 
\end{proof}

\begin{proposition}\label{invar}
The quantum invariant of a mapping torus $M_{A}$ of 
$A= \left(\begin{array}{cc}
\alpha & \beta \\
\gamma & \delta \\
\end{array}
\right)\in SL(2,\Z)$ 
is given by the formula 
 \begin{equation}\label{RTinvar}
RT_{\mathcal C}(M_{A})=\zeta^{-3\varphi(A)}
{\rm Tr}(\rho_{\mathcal C}^{\lambda,\zeta}(A))
\end{equation}\label{difer}
where $\varphi:SL(2,\Z)\to \Z$ is  the modified Meyer  function 
\begin{equation}
3\varphi(A)=\Phi_R(A)-3 {\rm sgn}\left(\gamma(\alpha+\delta-2)\right)
\end{equation}
\end{proposition}
\begin{proof}
The main reason to introduce the almost linear representations 
$f_{\mathcal C}^{\lambda,L}$ is the 
following result of Turaev (see \cite{Tu}, section IV.7,  Thm.7.2.1, p.209)  
which expresses the quantum invariant of a mapping torus as follows: 

\begin{proposition}\label{quant1}
Let $M_{h}$ be the mapping torus of some 
homeomorphism whose mapping class 
is $h\in{\mathfrak M}_g$. Then 
\begin{equation}
RT_{\mathcal C}(M_{h})= \zeta^{3\mu(\Lambda(h_*),L\oplus h_*(L),{\rm Diag})}{\rm Tr}(f_{\mathcal C}^{\lambda,L}(h))
\end{equation}
where $\mu$ is the Maslov index of the Lagrangian subspace of 
$-H_1(\Sigma_g;\R)\oplus H_1(\Sigma_g;\R)$ endowed with the symplectic form 
$-\omega\oplus \omega$, $\Lambda(h_*)$ denotes 
the graph of $h_*$, i.e. the subspace of 
vectors $x\oplus h_*(x)$, where $x\in H_1(\Sigma_g;\R)$ 
and ${\rm Diag}$ is the diagonal subspace $\Lambda({\rm 1}_{H_1(\Sigma_g;\R)})$. 
\end{proposition} 

Observe that the manifold $M_{h}$ and its invariant $RT_{\mathcal C}(M_{h})$ do not depend on the choice of the Lagrangian $L$, although 
$f_{\mathcal C}^{\lambda,L}(h)$ does. 

Now it suffices to check that 
\begin{equation}
\mu(\Lambda(A),L_0\oplus A(L_0),{\rm Diag}) 
= -{\rm sgn}\left((\alpha+\delta-2)\gamma\right)
\end{equation}
when $A\in SL(2,\Z)$. 
If $\gamma=0$ then one verifies that the Maslov index is $0$. 
Suppose now that $\gamma\neq 0$. 
A direct inspection shows that 
$(\Lambda(A) + L_0\oplus A(L_0))\cap {\rm Diag}$ is 
the one-dimensional subspace generated by the vector 
$e=(\alpha-1,\gamma)\oplus (\alpha-1,\gamma)$.  
The quadratic form associated to $e$ has value  
$\omega(e_2, e)$, where $e=e_1+e_2$ is any decomposition 
with $e_1\in \Lambda(A)$ and $e_2\in L_0\oplus A(L_0)$. 
We can take $e_1=(0,\gamma)\oplus (\beta\gamma,\delta\gamma)$ and 
$e_2=(\alpha-1,0)\oplus (\alpha(1-\delta),
\gamma(1-\delta))$. This implies that 
\begin{equation}
\omega(e_2, e)=(2-\alpha-\delta)\gamma
\end{equation}
Now the signature of this quadratic form is the value of the 
Maslov index and the formula above follows. 
\end{proof}

Since the orientation preserving 
homeomorphism type of the manifolds $M_A$ depends only on the conjugacy class of $A$  we obtain immediately the following property of Meyer's 
function:  

\begin{corollary}
Meyer's function $\varphi$ is conjugacy invariant. 
\end{corollary}

\begin{remark}
There exist a slight difference between the usual Meyer's function 
$\varphi_M$ from \cite{KM} and the modified Meyer function $\varphi(M)$ considered by here, following Turaev. This does not makes a big difference since 
it only affects the invariants for $M_A$ where $A$ is parabolic.  
Specifically we have: 
\begin{equation}
\varphi(A)-\varphi_M(A)=\left\{\begin{array}{ll}
\frac{1}{2}(1+{\rm sgn}(\delta)){\rm sgn}(\beta), &{\rm if }\, \gamma=0 \\
0, & {\rm otherwise }\\
\end{array}\right. 
\end{equation}
However the function $\phi-\phi_M$ is an integral 1-cocycle so that 
the boundary $\delta(\phi)$ and $\delta(\phi_M)$ are cohomologous. 
Notice that $\delta(\phi_M)$ is Meyer's signature 2-cocycle 
(see \cite{At,Me}).   
\end{remark}

\subsection{Reshetikhin-Turaev invariants vs Turaev-Viro invariants}
\begin{proposition}\label{RTV}
For any modular tensor category ${\mathcal C}$ we have the identity: 
\begin{equation}
RT_{\mathcal C}(M_A\#\overline{M_A})=TV_{\mathcal C}(M_A)
\end{equation}  
\end{proposition}
\begin{proof}
Since $RT_{\mathcal C}$ behaves multiplicatively with respect to connected sums 
we have 
\begin{equation}
RT_{\mathcal C}(M_A\#\overline{M_A})=RT_{\mathcal C}(M_A)\cdot 
RT_{\mathcal C}(\overline{M_A})=RT_{\mathcal C}(M_A)\cdot RT_{\overline{\mathcal C}}(M_A)
\end{equation} 
from (\cite{Tu}, II.2, (2.5.a)). Here $\overline{\mathcal C}$ 
denotes the mirror category of the 
modular tensor category $\mathcal C$ according to (\cite{Tu}, I.1.4). 
It is known that the rank $\lambda=\lambda_{\mathcal C}$ for $\mathcal C$ is 
equally a rank $\lambda=\lambda_{\overline{\mathcal C}}$ 
for $\overline{\mathcal C}$, although the roles of $p_{\mathcal C}^{\pm}$ 
are inverted, namely we have  
$p_{\mathcal C}^{+}=p_{\overline{\mathcal C}}^{-}$ and 
$p_{\mathcal C}^{-}=p_{\overline{\mathcal C}}^{+}$. 
It follows also that 
$\dim_{\mathcal C} i=\dim_{\overline{\mathcal C}} i$, for every $i\in I$, 
but $\omega_{{\mathcal C},i}^{-1} =\omega_{\overline{\mathcal C},i}$.
Furthermore the anomalies $\zeta_{\mathcal C}^3=p_{\mathcal C}^+/\lambda_{C}$ 
are inverse to each other, namely 
 $\zeta_{\overline{\mathcal C}}^{3}=
\zeta_{\mathcal C}^{-3}$. 
Therefore  
The $S$ matrix associated to the mirror category 
has  its entries  ${S({\overline{\mathcal C}})}_{ij}$ equal to 
${S({{\mathcal C}})}_{i^*j}$ (from \cite{Tu}, II.1.9, Ex. 1.9.(2)). 
At the last the matrix ${T({\overline{\mathcal C}})}_{ij}$ is 
the inverse of ${T({{\mathcal C}})}_{ij}$ since 
${T({\overline{\mathcal C}})}_{ij}=\omega_{{\mathcal C},i}^{-1}\delta_{ij}$. 

On the other hand we have the representation 
$\widetilde{\rho_{\mathcal C}^{\lambda,\zeta}}(x)=
\rho_{\mathcal C}^{\lambda,\zeta}(JxJ^{-1})$ defined above. 
We have the following identities, where $\zeta$ stands for $\zeta_{\mathcal C}$: 
\begin{equation}
\rho_{\overline{\mathcal C}}^{\lambda,\zeta^{-1}}({\mathfrak s}) = 
\lambda^{-1} \left(S({\mathcal C})_{i^*j}\right)=
\lambda^{-1} S({\mathcal C})^{-1}=
\widetilde{\rho_{\mathcal C}^{\lambda,\zeta}}({\mathfrak s})
\end{equation}
\begin{equation}
\rho_{\overline{\mathcal C}}^{\lambda,\zeta^{-1}}({\mathfrak t}) = 
\zeta T({\mathcal C})^{-1}= 
\widetilde{\rho_{\mathcal C}^{\lambda,\zeta}}({\mathfrak t})
\end{equation}
Therefore the two representations agree on every element 
\begin{equation} 
\rho_{\overline{\mathcal C}}^{\lambda,\zeta^{-1}}(x)=
\widetilde{\rho_{\mathcal C}^{\lambda,\zeta}}(x),\, {\rm for \;  any }\,\;  x\in SL(2,\Z)
\end{equation}

Now  using Proposition \ref{invar} and Lemmas \ref{double} 
and \ref{anomalyfree}  we obtain the identities: 
\begin{equation}
RT_{\mathcal C}(M_A)\cdot RT_{\overline{\mathcal C}}(M_A)=
{\rm Tr}(\rho_{{\mathcal C}}^{\lambda,\zeta}(A)) 
{\rm Tr}(\widetilde{\rho_{\mathcal C}^{\lambda,\zeta}}(A))=
{\rm Tr}(\rho_{{\mathcal C}}^{\lambda,\zeta}\otimes 
\widetilde{\rho_{\mathcal C}^{\lambda,\zeta}}(A))
={\rm Tr}(\rho_{D({\mathcal C})}^{\lambda, 1}(A)=TV_{\mathcal C}(M_A)
\end{equation} 

A more direct proof of  Proposition \ref{RTV} comes from 
the recent proof by Turaev and Virelizier (see \cite{TV}) 
of the formula $|M|_{\mathcal C}=RT_{D(\mathcal C}(M)$ for any oriented 
3-manifold, 
 and spherical fusion category of non-zero dimension ${\mathcal C}$. 
Here $|M|_{\mathcal C}$ is the simplicial 6j-symbol 
state sum defined in (\cite{Tu}, section 4). According to 
(\cite{Tu}, section IV,Thm.4.1.1) we have 
$|M|_{\mathcal C}=RT_{\mathcal C}(M)RT_{\mathcal C}(\overline{M})$. 

\end{proof}

\begin{corollary}\label{TV}
There exist pairs of matrices $A$, $B$ such that 
$M_A$ and $M_B$ are not homeomorphic but 
$|RT_{\mathcal C}(M_A)|=|RT_{\mathcal C}(M_B)|$ for any Hermitian 
modular tensor category ${\mathcal C}$. 
\end{corollary}
\begin{proof}
It is known that $RT_{\mathcal C}(\overline{M})=\overline{RT_{\mathcal C}(M)}$,  
for any Hermitian modular tensor category ${\mathcal C}$ (\cite{Tu}, II.5, 
Thm. 5.4) so the previous 
Proposition gives us  
$|RT_{\mathcal C}(M)|^2=TV_{\mathcal C}(M)$. 
Also the TQFT associated to  
${D(\mathcal C)}$ is anomaly-free. This 
follows from the fact that the  mapping class group representations 
in genus $g$ associated to $\mathcal C$ and $D(\mathcal C)$ satisfy 
$\rho_{g, D(\mathcal C)}=\rho_{g, \mathcal C}\otimes 
\widetilde{\rho}_{g, \mathcal C}$, in every genus $g$. In fact  
the right hand side tensor product is well-defined even when 
we have only projective representations and this shows that 
$\rho_{D(\mathcal C)}$ is a genuine linear representation so that 
the associated TQFT is anomaly-free.  

Eventually, 
the so-called Vafa's theorem (see \cite{Va,Et}) shows that the anomaly 
$\zeta^3$ of the TQFT
associated to $D(\mathcal C)$ is a root of unity for every modular 
tensor category $\mathcal C$ (actually it is enough  
to know that $|\zeta^3|=1$) 
and hence the associated invariants verify the claim.   
\end{proof}

\subsection{Congruence subgroups}

We will prove now: 
\begin{proposition}\label{modular}
Suppose that some modular 
representation $\rho_{\mathcal C}^{\lambda, \zeta}$ associated to the 
modular tensor category $\mathcal C$ factors through $SL(2,\Z/N\Z)$.
Let $A$ and $B$ be two integral matrices whose reductions mod $N$ 
are conjugate. If $\varphi(A)=\varphi(B)$ 
then $RT_{\mathcal C}(M_A)=RT_{\mathcal C}(M_B)$.  

Henceforth, if $A$ and $B$ are conjugate in {\em every} congruence quotient 
of $SL(2,\Z)$ then  $RT_{\mathcal C}(M_A)=RT_{\mathcal C}(M_B)$ for any 
modular tensor category $\mathcal C$ with the congruence property 
if and only if 
\begin{equation}
\varphi(A)=\varphi(B)
\end{equation} 
\end{proposition}
\begin{proof}
According to the Proposition \ref{invar} the invariant 
$RT_{\mathcal C}(M_A)$ is the trace of the endomorphism 
$\rho_{\mathcal C}^{\lambda, \zeta}(A)$ up to the factor $\zeta^{-3\varphi(A)}$. 
By hypothesis 
$\rho_{\mathcal C}^{\lambda, \zeta}$ factorizes as 
 $\rho_{\mathcal C}^{\lambda, \zeta}=
\rho_{\mathcal C}^{\lambda, \zeta,N}\circ \nu_N $, where 
$\nu_N:SL(2,\Z)\to SL(2,\Z/N\Z)$ is  the homomorphism of 
reduction mod $N$.

Since $\nu_N$ is surjective (see Lemma \ref{surj}) there exists 
$T\in SL(2,\Z)$ such that $\nu_N(A)=\nu_N(T^{-1}BT)$. 
Therefore 
\begin{equation}
{\rm Tr}(\rho_{\mathcal C}^{\lambda, \zeta}(A))=
{\rm Tr}\left(\rho_{\mathcal C}^{\lambda, \zeta,N}(\nu_N(T))\cdot 
\rho_{\mathcal C}^{\lambda, \zeta,N}(\nu_N(B))\cdot 
\left(\rho_{\mathcal C}^{\lambda, \zeta,N}(\nu_N(T))\right)^{-1}\right)=
{\rm Tr}(\rho_{\mathcal C}^{\lambda, \zeta}(B)) 
\end{equation}
Thus we have equality of quantum invariants of 
$M_A$ and $M_B$ if and only if $\zeta^{-3\varphi(A)}=\zeta^{-3\varphi(B)}$. 
Since there are modular categories whose anomaly $\zeta^3$ is a root of unity 
of arbitrary large degree the claim follows. 
\end{proof}

\begin{remark}
Eholzer conjectured in (\cite{E} that modular 
representations $\rho_{\mathcal C}^{\lambda, 1}$ have the congruence property 
for every RCFT, or in somewhat equivalent terms, for every modular tensor 
category. If Eholzer's conjecture were true, 
then  Proposition \ref{modular} would 
imply that $RT_{\mathcal C}(M_A)=RT_{\mathcal C}(M_B)$ for every 
modular tensor category ${\mathcal C}$ if $A$ and $B$ 
are conjugate in every congruence quotient of $SL(2,\Z)$ and 
$\varphi(A)=\varphi(B)$.  
\end{remark}
\begin{remark}
Notice that Ng and Schauenburg proved in \cite{NS} that the projective 
representation $\overline{\rho}_{\mathcal C}$ has the congruence property 
for every  modular tensor category ${\mathcal C}$. However this does not 
imply that some linear lift  $\rho_{\mathcal C}^{\lambda, \zeta}$  of it 
has also the congruence property (see \cite{NS}, section 7). 
Moreover, it is not clear whether the fact that  
$\rho_{\mathcal C}^{\lambda, \zeta}$ has the congruence property 
for one particular value of $(\lambda,\zeta)$ would imply that all 
modular representations $\rho_{\mathcal C}^{\lambda, \zeta}$ do have it.  
\end{remark}

\begin{example}
Consider now the torus bundles  $M_A$ and $M_B$, where 
$A=\left(\begin{array}{cc}
188 & 275 \\
121 & 177 \\
\end{array}
\right)$ and $B=\left(\begin{array}{cc}
188 & 11 \\
3025 & 177 \\
\end{array}
\right)$.  
We proved above that the absolute values of all TQFT 
invariants agree on $M_A$ and $M_B$. 
In order to find the phase factors of these invariants we have to compute 
$\varphi(A)$ and $\varphi(B)$. 
We use the reciprocity law for Dedekind sums which reads:
\begin{equation}
12s(\alpha,\gamma)+12s(\gamma,\alpha)=\frac{\gamma}{\alpha}+\frac{\alpha}{\gamma}+\frac{1}{\alpha\gamma}-3{\rm sgn}(\alpha\gamma)
\end{equation}
and the obvious $s(\alpha,\gamma)=s(\alpha',\gamma)$ if 
$\alpha\equiv\alpha'({\rm mod}\, \gamma)$. 
We find therefore 
\[ \varphi(A)=-1, \,\varphi(B)=-21\]
\end{example}

\section{Proof of Proposition \ref{classif}}
Recall from the Introduction that 
\begin{equation} 
\Gamma_A=\langle t,a,b | ab=ba, tat^{-1}=a^{\alpha_{11}}b^{\alpha_{12}}, 
tbt^{-1}=a^{\alpha_{21}}b^{\alpha_{22}}\rangle 
\end{equation} 
where 
$A\in SL(2,\Z)$ is a matrix  with entries $\alpha_{ij}$, 
$1\leq i,j\leq 2$, such that  ${\rm Tr}(A)\neq 2$. 
We have the  following exact sequence: 
\begin{equation}
1\to \Z^2\stackrel{i_A}{\to} \Gamma_A\stackrel{p_A}{\to} \Z\to 1
\end{equation}
defined by 
\begin{equation}
i_A(1,0)=a, i_A(0,1)=b, p_A(t)=1, p_A(a)=p_A(b)=0
\end{equation}
\begin{proposition}
The abelian subgroup $i_A(\Z^2)\subset \Gamma_A$ is 
the radical  set $R_A$ of $[\Gamma_A,\Gamma_A]$ 
in $\Gamma_A$, namely the set of those 
$x\in \Gamma_A$ for which there exists some $k\neq 0$ 
such that $x^k\in [\Gamma_A,\Gamma_A]$.  
\end{proposition}
\begin{proof}
Consider $x\in R_A$. By the definition of the radical 
set  there exists $k\neq 0$ such that 
$x^k\in [\Gamma_A,\Gamma_A]$ and thus the image of $x^k$ 
vanishes  in every abelian quotient of $\Gamma_A$. 
This implies that $p_A(x^k)=0$. Since $k\neq 0$ we have 
$p_A(x)=\frac{1}{k}p_A(x^k)=0$, which means that 
$x\in \ker p_A =i_A(\Z^2)$.

\begin{lemma}
Every element of $\Gamma_A$ can be uniquely 
written in the form $t^sa^nb^m$. 
\end{lemma}
\begin{proof}
For every $x\in \Z^2$ the conjugacy by the stable letter $t$  
can be expressed as follows: 
\begin{equation}
t i_A(x) t^{-1}=i_A(A(x)),\;  t^{-1} i_A(x) t= i_A(A^{-1}(x))
\end{equation}  
where $A(x)$ denotes the left multiplication by the 
matrix $A$ of the vector $x\in \Z^2$. 

Consider now a word in the generators containing at least one letter $t$.
We use the conjugacy relations above to move to the 
left every occurrence of the letter $t$ (or $t^{-1}$). 
If a leftmost subword of the new word is of the
 form $i_A(x)t^{\varepsilon}$, with 
non-zero $x\in \Z^2$ then rewrite it as  
$t^{\varepsilon}(t^{-\varepsilon}i_A(x)t^{\varepsilon})=
t^{\varepsilon}\cdot i_A(A^{-\varepsilon}(x))$ and continue. 
This process will stop eventually because there are only finitely many 
occurrences of $t$ and the resulting word will have the desired form. 

For the uniqueness it suffices to see that $\Gamma_A$ is a 
HNN extension with one stable letter and  conclude by the classical 
Britton's Lemma. 
\end{proof}

\begin{lemma}
Let  $U$ denote  the integral matrix  $A-{\mathbf 1}$. 
Then $[\Gamma_A,\Gamma_A]$ is the subgroup $i_A(U(\Z^2))$ of $R_A$. 
\end{lemma}
\begin{proof}
If $x\in \Z^2$ then we have the identities:  
\begin{equation}
[t,i_A(x)]=t i_A(x)t^{-1}i_A(x^{-1})=i_A(A(x)-x) 
\end{equation} 
This shows that $i_A(U(\Z^2))\subset [\Gamma_A,\Gamma_A]$. 
 
Conversely, let us consider $u\in i_A(U(\Z^2))$.
Then $tut^{-1}\in i_A(U(\Z^2))$ and $t^{-1}ut\in i_A(U(\Z^2))$ because
\begin{equation}
t i_A(A(x)-x)t^{-1} =i_A(A^2(x)-A(x))=i_A(A(y)-y), {\rm where }\; y=Ax
\end{equation} 
and 
\begin{equation}
t^{-1} i_A(A(x)-x)t =i_A(x-A^{-1}(x))=i_A(A(y)-y), {\rm where }\; y=A^{-1}x
\end{equation} 
Further  $p_A([\Gamma_A,\Gamma_A])=0$ so that 
$[\Gamma_A,\Gamma_A]$ is a subgroup of $R_A$, in particular it is abelian.  
Thus the action of $a$ and $b$ (or any $x\in i_A(\Z^2)$) 
by conjugacy on $i_A(U(\Z^2))$ is trivial.

Now $[\Gamma_A,\Gamma_A]$ is generated by 
the commutators of elements of $\Gamma_A$.
The Hall identities 
\begin{equation}
[xy, z]= y^{-1} [x,z]y \cdot  [y,z], \; 
[x,yz]=[x,z]\cdot  z^{-1} [x,y] z 
\end{equation}
Then a double recurrence on the number of letters in the words representing 
$x,y\in \Gamma_A$ and the previous observations about the conjugacy
by generators prove that 
$[x,y]\in i_A(U(\Z^2))$.  
\end{proof}

\begin{lemma}
If $A$ has no eigenvalue equal to $1$ then $[\Gamma_A,\Gamma_A]$ 
is a finite index  subgroup of $R_A$. 
\end{lemma}
\begin{proof}
Since ${\rm Tr}(A)\neq 2$ we have  $|\det(U)|=|{\rm Tr}(A)-2|\neq 0$ so that 
$U\Z^2\subset \Z^2$ is a subgroup of index $|\det(U)|$. Since $i_A$ 
is an isomorphism the Lemma follows. 
\end{proof}  
In particular, for every $x\in i_A(\Z^2)$ there exists some $k$ (which divides 
$|\det(U)|$) for which $x^k\in [\Gamma_A,\Gamma_A]$, as claimed.  
\end{proof}

Now any isomorphism $\phi:\Gamma_A\to \Gamma_B$ should restrict to 
an isomorphism $\phi:i_A(\Z^2)\to i_B(\Z^2)$. In fact, if $x\in R_A$ there 
exists $k\neq 0$ such that $x^k\in [\Gamma_A,\Gamma_A]$ so that 
$\phi(x)^k\in [\Gamma_B,\Gamma_B]$, meaning  that 
$\phi(x)\in R_B$. Now isomorphisms $\phi$ between free 
abelian groups are determined by some invertible matrix, namely 
\begin{equation}
\phi(i_A(x))=i_B(V(x)), {\rm for}\; x\in \Z^2
\end{equation}
where $V\in GL(2,\Z)$. 

\begin{lemma}
There exists $E\in \Z^2$ such that 
either $\phi(t)=ti_B(E)$ or $\phi(t)=t^{-1}i_B(E)$. 
\end{lemma}
\begin{proof}
In fact  $\phi$ induces an isomorphism 
$\phi_*=\Gamma_A/R_A\to \Gamma_B/R_B$. 
Now both groups $\Gamma_A/R_A$ and  $\Gamma_B/R_B$ 
are isomorphic to $\Z$ and so $\phi_*$ is 
$\varepsilon {\bf 1}_{\Z}$ where $\varepsilon\in\{-1,1\}$. 
This is precisely the claim of the Lemma. 
\end{proof}

In order to get rid of the translation factor in $\phi$ 
we need the following extension result:  
\begin{lemma}
For every $E\in \Z^2$ there exists an automorphism 
$L_E:\Gamma_B\to \Gamma_B$ 
such that 
\begin{equation}
L_E(ti_B(x))=ti_B(x+E), L_E(i_B(x))=i_B(x), \; {\rm  for\, every} \;
x\in \Z^2
\end{equation}
\end{lemma}
\begin{proof}
We have to show that the homomorphism defined on the generators by:  
\begin{equation}
L_E(t)=ti_B(E), L_E(a)=a,  L_E(b)=b, 
\end{equation}
is well-defined. First we compute:  
\begin{equation}
L_E(t^{-1}) =(ti_B(E))^{-1}= i_B(-E)t^{-1}=t^{-1}\cdot ti_B(-E)t^{-1}=
t^{-1}i_B(-B(E))
\end{equation}
It suffices  now to verify that  
the relations in $\Gamma_B$ are preserved, namely at first:   
\begin{equation}
L_E(txt^{-1})=ti_B(x +E)t^{-1}i_B(-B(E))=
i_B(B(x+E))i_B(-B(E))=i_B(B(x))
\end{equation}
for $x\in \Z^2$, and second $L_E(ab)=L_E(ba)$, which is obvious. 
Hence $L_E$ defines a homomorphism, whose inverse is $L_{-E}$ 
which implies that $L_E$ is an automorphism of $\Gamma_B$. 
An immediate computation shows that  
\begin{equation}
L_C(t^si_B(x))=
\left\{\begin{array}{ll}
t^si_B(x+E + B(E)+B^2(E) +\cdots + B^{s-1}(E)), &{\rm if }\; s\geq 1\\
x, &{\rm if }\; s=0\\
t^si_B(x-B^{-1}(E)-B^{-2}(E)-\cdots - B^{-s}(E)), & {\rm if }\; s\leq -1
\end{array}
\right.
\end{equation}  
This proves the Lemma. 
\end{proof}

We replace  now  the isomorphism $\phi$ by  the composition 
$L_{-E}\circ \phi:\Gamma_A\to \Gamma_B$, which 
has trivial translation part and keep the same notation 
$\phi$ for the new isomorphism which has the property that 
\begin{equation}
\phi(t)=t^{\varepsilon}, {\rm where }\; \phi_*=\varepsilon {\bf 1}_{\Z}
\end{equation}
Recall now that $ti_A(x)t^{-1}=i_A(A(x))$, for any $x\in \Z^2$. 
If $\varepsilon=1$ then on one hand we have 
\begin{equation}
\phi(ti_A(x)t^{-1})=t\phi(i_A(x))t^{-1}=t i_B(V(x)t^{-1}=
i_B(BV(x))
\end{equation} 
and on the other hand:  
\begin{equation}
\phi(ti_A(x)t^{-1})=\phi(i_A(A(x))=i_B(VA(x))
\end{equation} 
The two right hand side terms from above must coincide, 
so $i_B(VA(x))=i_B(BV(x))$ for every $x\in i(\Z^2)$, 
which implies $VA=BV$ so that $A$ and $B$ are conjugate 
within $GL(2,\Z)$.

\begin{lemma}
There is an automorphism 
$J:\Gamma_B\to \Gamma_{B^{-1}}$ given by 
$J(t)=t^{-1}$, $J(a)=a$, $J(b)=b$. 
\end{lemma}
\begin{proof}
Clear, by direct computation. 
\end{proof} 

Assume now that  $\varepsilon=-1$. We consider then 
the isomorphism $J\circ\phi:\Gamma_A\to \Gamma_{B^{-1}}$, which 
satisfies:
\begin{equation}
J\circ\phi(t)=t, \; J\circ\phi(i_A(x))=i_B(V(x)),\,  
{\rm for}\; x\in \Z^2
\end{equation}
The argument from above shows  now that 
$A$ and $B^{-1}$ are conjugate by the 
matrix $V\in GL(2,\Z)$. 
This proves Proposition \ref{classif}. 

\begin{remark}
The result holds more generally when  
$A$ and $B\in GL(2,\Z)$ and  $|{\rm Tr}(A)|\neq 2 \neq|{\rm Tr}(B)$, with the same 
proof. 
\end{remark}

\section{Proof of Proposition \ref{noteq}}

\begin{proposition}\label{localeq}
Let $A=\left(\begin{array}{cc}
1 & kq^2 \\
kv & 1+k^2q^2v \\
\end{array}
\right)$ and $B=\left(\begin{array}{cc}
1 & k \\
kvq^2 & 1+k^2q^2v \\
\end{array}
\right)$ denote matrices from $SL(2,\Z)$, where $k\in\Z$,  
$q$ is an odd prime number  $q\equiv 1({\rm mod})\, 4)$,  
$v$ is a positive integer such that $-v$ is a non-zero quadratic residue 
mod $q$ which is  divisible either 
by a prime $p\equiv 3({\rm mod}\, 4)$, or by $4$. 
Then the following hold: 
\begin{enumerate}
\item For every natural $N$ there exists $T_N\in SL(2,\Z)$ such that 
$\nu_N(A)$ and $\nu_N(B)$ are conjugate by the 
matrix $\nu_N(T)$, where $\nu_N:\Z\to \Z/N\Z$ is the reduction mod $N$.
\item The matrix $A$ is conjugate neither to $B$ nor to $B^{-1}$ within $GL(2,\Z)$.
\end{enumerate}
\end{proposition}
\begin{proof}
The conjugacy condition $TA=BT$ is equivalent to a 
linear system of equations having the  
2-parameter family of solutions 
$T(x,y)=\left(\begin{array}{cc}
x & y \\
v y & q^2x+ kq^2vy \\
\end{array}
\right)$. 
The matrix $T$ belongs to $SL(2,\Z)$ if and only if  
$x, y\in\Z$ and the determinant of $T$ is $1$, namely if and only if the 
quadratic Diophantine equation 
\begin{equation}\label{dioph1} 
q^2 x^2+ kq^2v xy-v y^2=1
\end{equation}
has integral solutions.

In a similar way $\nu_N(A)$ and $\nu_N(B)$ are conjugate by a matrix 
$T\in SL(2,\Z/N\Z)$ if and only if the equation (\ref{dioph1})
has solutions $x,y\in \Z/N\Z$. We can improve this last statement as follows:  

\begin{lemma}\label{surj}
The homomorphism $\nu_N:SL(m,\Z)\to SL(m,\Z/N\Z)$ is surjective. 
\end{lemma}
\begin{proof}
It is known that the group $SL(m,\Z/N\Z)$ is generated 
by the matrices of the form 
${\mathbf 1}+E_{ij}$, where $E_{ij}$ has only one non-zero entry, which is  
$1$, sitting in position $ij$. See (\cite{HM}, Thm. 4.3.9) for a proof. 

Here is an explicit construction when $m=2$. 
Let $U=\left(\begin{array}{cc}
u_{11} & u_{12} \\
u_{21} & u_{22} \\
\end{array}
\right)$ be an integral matrix whose reduction mod $N$ is a given 
matrix of $SL(2,\Z/N\Z)$.   
There exist integers $\alpha,\beta$ such that 
$\alpha u_{12}-\beta u_{11}={\rm g.c.d.}(u_{11},u_{12})$. 
Set then:
\[ T= \left(\begin{array}{cc}
u_{11} + N(\alpha+u_{11}\gamma) +1-{\det U} & 
u_{12}+ Nb + N(\beta+u_{12}\gamma)+1-{\det U} \\
u_{21}-\frac{1}{{\rm g.c.d}(u_{11},u_{12})}
u_{11}\gamma({\det U} -1) & u_{22}-
\frac{1}{{\rm g.c.d}(u_{11},u_{12})}u_{12}\gamma({\det U} -1)  \\
\end{array}
\right)
\]  
where $\gamma =\beta t_{21}-\alpha u_{22}$. 
Now $\nu_N(T)=\nu_N(U)$ because $\nu_N(\det U)\equiv 1\in \Z/N\Z$, 
and $\det T=1$ so that $T\in SL(2,\Z)$. 
\end{proof}

Therefore, if the Diophantine equation (\ref{dioph1}) has solutions in $\Z/q^s\Z$ 
for every prime $q$ then, for every natural $N$ 
there exists $T_N\in SL(2,\Z)$ such that 
$\nu_N(A)$ and $\nu_N(B)$ are conjugate by the matrix $\nu_N(T)$.

\begin{lemma}\label{modulo} 
If $-v$ is a non-zero quadratic residue mod $q$ then the equation (\ref{dioph1}) 
has solutions in $\Z/N\Z$ for every $N$. 
\end{lemma}
\begin{proof}
Let us show that this equation has solutions mod $p^l$ for every 
prime $p$ and positive integer $l$, which will imply that there exist 
solutions mod $N$ for every $N$.  

If $p\neq q$ then take $x=\overline{q}$ and $y=0$,  
where $\overline{a}$ denotes the inverse of $a$ mod $p^s$. 
If $p=q$ then $-v$ is also a quadratic residue 
mod $q^l$ for every positive $l$, by the quadratic reciprocity 
law and the fact that  $q\equiv 1({\rm mod}\, 4)$. 
Thus there exists  an invertible 
$z$ such that 
$-v\equiv z^2({\rm mod}\, q^l)$. Therefore 
$x=0$ and $y=\overline{z}$ is a solution mod $h^l$. 
\end{proof}
  
\begin{lemma}\label{noint}
If $q$ is an odd prime and $v$ is a positive integer 
then the equation $(\ref{dioph1})$ 
has not integral solutions. 
\end{lemma}
\begin{proof}
The discriminant is a perfect square $w^2$ such that 
\begin{equation}\label{dioph2}
w^2-(k^2q^4v^2+4q^2v) y^2=4q^2
\end{equation}

If $k=2n$ is even then $w$ is even and divisible by $q$ 
so  that we can put $w=2qu$, for some integer $u$ satisfying: 
\begin{equation}\label{dioph2b}
u^2-(n^2q^2v^2+v) y^2=1
\end{equation}

If $k$ is odd then $w=qu$ and the equation reads:
\begin{equation}\label{diopho}
u^2-(k^2q^2v^2+4v)y^2=4
\end{equation}

\begin{lemma}\label{mult}
If $(u,y)$ is an integer solution for either one of the equations 
(\ref{dioph2b}) or (\ref{diopho}) then $y$ is divisible by $q$. 
\end{lemma} 
\begin{proof}
Consider first $k$ even when the equation (\ref{dioph2b}) 
is a Pell equation. Let us remind briefly the theory of the Pell equation:
\begin{equation}\label{gpell}
u^2-Dy^2=1
\end{equation}
where $D$ is a positive integer, which is not a square. 
There exists only one minimal solution which can be constructed following 
classical results (see \cite{Mo,NZM}) as follows. We set 
\begin{equation}
P_0=0, Q_0=1, a_0=[\sqrt{D}]
\end{equation}
\begin{equation}
H_{-2}=1, H_{-1}=0, G_{-2}=-P_0, G_{-1}=Q_0
\end{equation} 
We define inductively:  
\begin{equation}
H_{i}=a_i H_{i-1}+H_{i-2}, G_{i}=a_i G_{i-1}+G_{i-2}
\end{equation}
\begin{equation}
P_{i}=a_{i-1}Q_{i-1}-P_{i-1}, Q_{i}=(D- P_{i}^2)/Q_{i-1} 
\end{equation}
\begin{equation}
a_i=\left[\frac{P_i +\sqrt{D}}{Q_i}\right]
\end{equation}
We have therefore 
\begin{equation}
G_{i-1}^2-DH_{i-1}^2=(-1)^i Q_i
\end{equation}
The algorithm for solving the Pell equation is as follows. 
Find  the smallest even integer $l\geq 1$ such that $Q_l=1$. 
Then $(G_{l-1}, H_{l-1})$ is the minimal  non-trivial solution 
$(u_0,y_0)$ to the Pell equation (\ref{gpell}). 

Moreover, any other  (positive) integral solution  can be obtained 
from the minimal one by means of the following recurrence: 
\begin{equation}
u_{s+1}=u_{0} u_s + D y_0 y_s, \, y_{s+1}= y_0 u_s  + x_0  y_s, \, {\rm for}\; s\geq 0
\end{equation}

The previous algorithm  (notice that $D$ is not a square) 
gives us the minimal solution for (\ref{dioph2b})
\begin{equation}
u_0=2n^2q^2v+1, \, y_0=2nq
\end{equation}
A recurrence on $s$ shows that $y_s$ is a multiple of $q$ 
for every $s\geq 0$. 

Assume now that $k$ is odd where the  
equation (\ref{diopho}) is a Pell-type equation.  

When $v$ is odd we can use the same algorithm  as used 
for the Pell equation  above to solve 
(\ref{diopho}) but starting from the initial data: 
\begin{equation}
P_0=1, \, Q_0=2
\end{equation}
because we are in the situation when $D\equiv 1({\rm mod}\, 4)$. 
Then the minimal solution is 
\begin{equation}
u_0=k^2q^2v+2, \, y_0=kq 
\end{equation}
and the same arguments show that all solutions $y$ are multiple of $q$. 

Eventually assume that $v$ is even, $v=2v'$, such that $u=2u'$ for some 
integer $u'$ and the equation (\ref{diopho}) becomes: 
\begin{equation}\label{diophoo}
u'^2-(k^2q^2v'^2+2v')y^2=1
\end{equation} 
One finds the minimal solutions 
\begin{equation}
u'_0=k^2q^2v'+1, \, y_0=kq 
\end{equation}
Thus all solutions $y$ are multiple of $y$. 
This proves Lemma \ref{mult}. 
\end{proof}

\begin{remark}
If $v$ is negative the minimal solutions are different, for instance 
when $v=-1$ and $k$ is even we have $u_0=\frac{k}{2}q$, 
$y_0=1$, so that the previous Lemma cannot be extended to negative $v$.  
\end{remark}

Going back to the original equation (\ref{dioph1}) if $y$ 
were a multiple of $q$ it would imply that $q$ divides $1$,  which is 
a contradiction. Thus (\ref{dioph1}) has not integral solutions and hence  
Lemma \ref{noint} is proved. 
\end{proof}

Further $-v$ was assumed to be a  quadratic residue modulo $q$. 
Thus Lemma \ref{modulo} shows that 
the equation (\ref{dioph1}) 
has solutions in $\Z/N\Z$ for every $N$ but has not integral solutions. 
In particular, the matrices $A$ and $B$ are not conjugate in $SL(2,\Z)$. 
In order to show that they are 
not conjugate into $GL(2,\Z)$ either it amounts to prove that 
the equation corresponding to $\det(T)=-1$, namely 
\begin{equation}\label{dioph3}
q^2x^2+2kq^2vxy-v y^2= -1
\end{equation}
has not integral solutions. 
If $v$ is divisible by a prime number $p$ 
which is congruent to $3$ mod $4$ then 
the reduction mod $p$ of the equation (\ref{dioph3}) 
reads $q^2x^2\equiv -1({\rm mod}\, p)$. 
But $-1$ is not a quadratic residue mod $p$ when 
$p$ is as above. The same argument works when $v$ is divisible 
by $4$. This shows that the matrices $A$ and $B$ are not conjugate 
in $GL(2,\Z)$.

Eventually, consider the conjugacy between $A$ and 
$B^{-1}=\left(\begin{array}{cc}
1+4k^2q^2v & -2k \\
-2kq^2v & 1 \\
\end{array}
\right)$. The linear equations $VA=B^{-1}V$ has the solutions 
$V(x,y)=\left(\begin{array}{cc}
x & y \\
-v y + 2kq^2v x & - q^2x\\
\end{array}
\right)$. The condition $\det(V)=\pm 1$ is actually 
the same couple of equations 
\begin{equation}
q^2x^2 +2kq^2v xy -v y^2 = \mp 1
\end{equation}
studied above. Therefore $A$ and $B^{-1}$ are not conjugate within $GL(2,\Z)$, as claimed. 
 \end{proof}

\begin{remark}
Another pair of matrices satisfying the claims from Proposition 
\ref{localeq} was obtained 
by Stebe in \cite{S}, as follows: 
\begin{equation}
A=\left(\begin{array}{cc}
188 & 275 \\
121 & 177 \\
\end{array}
\right), \,\, B=\left(\begin{array}{cc}
188 & 11 \\
3025 & 177 \\
\end{array}
\right) 
\end{equation}
\end{remark}

\begin{remark}
We have $\varphi(A)-\varphi(B)=2k(q^2-1)(v+1)$. 
Since $v$ is  positive all pairs $(A,B)$ furnished by Proposition \ref{localeq} have $\varphi(A)\neq \varphi(B)$ and hence the manifolds $M_A$ and $M_B$  
can be distinguished by their Reshetikhin-Turaev 
invariants. 
\end{remark}

\begin{remark}
There exist always rational solutions to the Diophantine equation above 
and thus the matrices $A$ and $B$ are always conjugate within $SL(2,\Q)$. 
This implies that the associated 3-manifolds $M_A$ and $M_B$ are 
commensurable.  
\end{remark}

\section{Proof of Theorem \ref{RTagree}}

\subsection{Outline of the proof}
According to the previous sections it suffices to show that 
here exist pairs of Anosov matrices $A$ and $B$ such that 
their images $A$ and $B$ are conjugate within $SL(2,\Z/m\Z)$ 
for every $m$, neither $A$ and $B$, nor $A$ and $B^{-1}$ are conjugate 
in $GL(2,\Z)$ (thus satisfying the claims of Proposition \ref{localeq})  
and moreover $\varphi(A)=\varphi(B)$.   
The key ingredient is to reformulate these requirements as follows:  

\begin{proposition}\label{reciprocal}
There exist infinitely many pairs of Anosov matrices $A$ and $B$ such that: 
\begin{enumerate}
\item $A$ and $B$ are conjugate in $SL(2,\Z/m\Z)$ 
for every $m$;
\item $A$ and $B$ are reciprocal, namely they are 
conjugate in $SL(2,\Z)$ to $A^{-1}$ and $B^{-1}$ respectively; 
\item $A$ and $B$ are inert, namely they are 
conjugate in $SL(2,\Z)$ to $wAw^{-1}$ and $wBw^{-1}$ respectively, 
where $w= \left(\begin{array}{cc}
1 & 0 \\
0  & -1\\
\end{array}\right)$;
\item $A$ and $B$ are not conjugate in $SL(2,\Z)$.    
\end{enumerate}
\end{proposition}

\vspace{0.2cm}
{\em Proof of Theorem \ref{RTagree} assuming Proposition \ref{reciprocal}}. 
If $A$ were conjugate to $B$ in $GL(2,\Z)$, namely 
$A=sBs^{-1}$, with $s\in GL(2,\Z)$ then $\det(s)=-1$ and 
$wAw^{-1}=ws B (ws)^{-1}$, with $\det(ws)=1$.   
Since $A$ is inert, this would imply that 
$A$ is conjugate in $SL(2,\Z)$ to $B$, which contradicts our assumption. 
Since $A$ and $B$ are reciprocal $A$ cannot be conjugate to $B^{-1}$ 
in $GL(2,\Z)$ either. 

Eventually recall that $\varphi$ is constructed from $\Phi_R$ in such a way 
that it becomes a quasi-homomorphism $\varphi: SL(2,\Z)\to \Z$. 
Namely, the following hold (see \cite{Me}) :

\begin{equation}
\varphi(CAC^{-1})=\varphi(A), \; {\rm for } \; C\in SL(2,\Z),
\end{equation}

\begin{equation}
\phi(A^{-1})=-\varphi(A)
\end{equation}

In particular, if $A$ and $B$ are reciprocal, then 
$\varphi(A)=\varphi(B)=0$ and this actually holds for any quasi-homomorphism 
$\varphi$. This will settle Theorem \ref{RTagree}.

 \subsection{Proof of Proposition \ref{reciprocal}}
Reciprocal (conjugacy) classes in $SL(2,\Z)$ 
were recently discussed by Sarnak in \cite{Sarnak}. 
Let $A^{\bot}$ denote the transpose of $A$. 
Since the transpose $(A^{-1})^{\bot}= \left(\begin{array}{cc}
0   & 1 \\
-1  & 0\\
\end{array}\right) A \left(\begin{array}{cc}
0   & 1 \\
-1  & 0\\
\end{array}\right)$, it follows that $A$ is reciprocal if and only if 
it is conjugate to its transpose $A^{\bot}$ (see also \cite{Sarnak}, p.218). 
Recall that $A$ is ambiguous if $A$ is conjugate within $SL(2,\Z)$ 
to $w^{-1}A^{-1}w$. 

We say that $A$ and $B$ are in the same genus if their images  
are conjugate within $SL(2,\Z/m\Z)$, for every $m$.  
Our aim is to find reciprocal and inert conjugacy classes 
in the same genus. 

Let $D$ be an odd (square-free) fundamental discriminant. 
Following Gauss (see \cite{BS}) there are 
$2^{\sigma(D)-1}$ genera of primitive integral binary 
forms, where $\sigma(D)$ is the number of distinct prime divisors of $D$.  

Denote by $\mathcal D^-$ the set of those D for which the negative 
Pell equation 

\begin{equation}
X^2- DY^2= -4
\end{equation}
has integral solutions. It is known that $D\in \mathcal D^-$ 
if and only if the narrow class group $C_D$ coincides with the 
class group $Cl_D$ of $\Q(\sqrt{D})$ (see \cite{St}, Lemma 2.1). 

Recall that the 4-rank of an abelian group $C$ is the rank of 
$C^2/C^4$ which counts the number of distinct cyclic factors of 
order 4. 

\begin{lemma}
Every $D\in \mathcal D^-$ such that 
the 4-rank of $C_D$ is non-trivial gives raise to a pair 
of non-conjugate reciprocal matrices in the same principal genus. 
\end{lemma}
\begin{proof}
According to Gauss (see \cite{BS}) the group of genera 
is $C_D/C_D^2$. The classes in the kernel of the projection 
$C_d\to C_d/C_D^2$ form  the principal genus. 
The set of ambiguous classes is identified with 
the kernel of the  square homomorphism $\delta: C_D\to C_D$ 
given by $\delta(x)=x^2$. 
Therefore the elements of order $2^n$ in $C_D$, with $n\geq 2$ 
are ambiguous classes in the principal genus. 

When $D\in \mathcal D^-$ it is known that every class is inert and 
every ambiguous class is reciprocal and viceversa (see \cite{Sarnak}, p.214). 
In particular, if the 4-rank of $C_D$ is positive then there are at least two 
inert and reciprocal classes in the principal genus. They are non-conjugate 
as they are distinct classes in $C_D$.  
\end{proof}

There exists a simple method developped by R\'edei and Reichardt 
(see \cite{RR}) to find the 4-rank of the narrow class group $C_D$. 
Let $D=p_1p_2\cdots p_n$ be the decomposition in odd prime numbers 
of $D$.  The R\'edei matrix $M_D$ is the $n$-by-$n$ matrix  over 
$\Z/2\Z$ whose  entries $a_{ij}$ are: 
\begin{equation}
a_{ij}=\left\{ \begin{array}{ll}
1, & {\rm if }\; i\neq j, \; {\rm and} \; \left(\frac{p_i}{p_j}\right)=-1\\
0, & {\rm if }\; i\neq j, \; {\rm and} \; \left(\frac{p_i}{p_j}\right)=1\\
\end{array}\right.  
\end{equation}
\begin{equation}
a_{jj}=\sum_{i\neq j, 1\leq i\leq n}a_{ij}
\end{equation}
Here $\left(\frac{p}{q}\right)\in\{-1,1\}$ is the Legendre symbol, 
equal to $1$ if and only if $p$ is a quadratic residue mod $q$. 
Eventually, following (\cite{RR})  
the 4-rank of $C_D$ is given by $\sigma(D)-1-{\rm rank}_{\Z/2\Z} M_D$. 

We will consider from now on $D$ of the form 
$D=u^2+4$, so that the negative Pell equation has obvious solutions 
$X=u$, $Y=1$.  We seek for those $D$ which are odd square-free and 
such that the associated R\'edei matrix is identically zero. 
If $D$ has at least two prime factors then the 4-rank of $C_D$ is non-trivial. 
In this case it is easy to find explicit matrices $A$ and $B$   
corresponding to ambiguous, inert and reciprocal pairs of classes 
in the principal genus. The trace $t$ of $A$ will be 
\begin{equation}
t=D-2
\end{equation}
so that it verifies 
\begin{equation}
t^2-Du^2=4
\end{equation}
For each positive integral solutions
$(a,b)$ of the equation $4a^2+b^2=D$ 
we have associated the classes of binary forms $(a,b,-a)$, which correspond 
to the symmetric matrices:
\begin{equation}
A_{a,b}=\left( \begin{array}{cc}
\frac{1}{2}\left(t-ub\right)  & a u \\
au & \frac{1}{2}\left(t+ub\right)\\
\end{array}
\right)
\end{equation}
These are obviously reciprocal classes in the principal 
genus $C_D^2$ of $C_D$. The examples in Theorem \ref{RTagree} 
arise when choosing $u\in\{21, 51,53,55\}$ for which 
$\sigma(D)=2$ and $M_D=0$. 

Eventually, there are infinitely many $D\in\mathcal D^-$ 
for which the 4-rank of $C_D$ is positive. 
Let $\mathcal D$ denote the set of special discriminants, namely 
the set of those $D$ whose prime factorization has only 
distinct odd primes of the form $p\equiv 1 ({\rm mod}\; 4)$ 
and possibly $8$.  Then in (\cite{FouK}, Theorem 2) the authors state
that the subset of those $D\in\mathcal D$ for which the 
4-rank of $C_D$ equals 1 and the 8-rank vanishes (and 
hence $D\in\mathcal D^-$) has positive density within 
the set $\mathcal D$. In particular this set is infinite.

\begin{remark}
We conjecture that the number of distinct cyclic factors 
of order  $2^m\geq 4$ of the class group $C_D$, where  
$D$ runs over the odd square-free $D$ of the form $n^2+4$ 
is unbounded. 
\end{remark}

\section{Proof of Theorem \ref{comm}}

\subsection{Abelian invariants}
We will consider the $U(1)$ gauge theory as defined in 
\cite{F1,F2,Go,MOO} and  then generalized in \cite{Del1,Del2}.   
One chooses a root of unity $q$  of order $k$ for odd $k$ 
and of order $2k$ for even $k$. Then in \cite{MOO} there is defined 
the invariant $Z_k(M,q)$ for 3-manifolds $M$ as follows.
Set $L$ be a framed link  with $n$ components in $S^3$ such that the 3-manifold $M$ 
is obtained by Dehn surgery on $L$. Let $A_L$ 
denote the linking matrix of $L$. 
We define then after (\cite{MOO}, (1.1)) the {\em MOO invariant} 
of the 3-manifold $M$ as being: 
\begin{equation}\label{inv}
Z_k(M,q)=\left(\frac{G_k(q)}{|G_k(q)|}\right)^{-\sigma(A_L)}
|G_k(q)|^{-n}\sum_{x\in (\Z/k\Z)^n} q^{^TxA_Lx}
\end{equation}
where $\sigma$ denotes the signature of the matrix and  
the Gaussian sums are given by: 
\begin{equation}
G_k(q)=\sum_{h\in \Z/k\Z} q^{h^2}
\end{equation}
Notice that for even $k$ the value of  $q^{^TxA_Lx}$ is defined by 
taking arbitrary lifts $\tilde{x}\in (\Z/2k\Z)^n$ and setting 
$q^{^TxA_Lx}=q^{^T\tilde xA_L\tilde x}$, which is independent on 
the choice of the lifts, since $A$ is symmetric.  

These invariants where further extended by Deloup in \cite{Del1}
by making use of general quadratic forms and finally extended to TQFTs 
in \cite{Del2}. These TQFT correspond to suitable modular tensor 
categories, which are related to  the Drinfeld double 
$D(\Z/k\Z)$ of the finite group $\Z/k\Z$ and to the geometric $U(1)$ 
Chern-Simons gauge theories. A more precise statement is given in 
(\cite{Del1}, Appendix A) where the invariants $Z_k$ 
and their generalizations are identified with the 
Reshetikhin-Turaev invariants associated 
to a modular category ${\mathcal A}$ coming from an abelian group, which 
is  described by Turaev in (\cite{Tu}, p.29). 
 
The Turaev-Viro invariants invariants $TV_{\mathcal A}$ are therefore 
the absolute values of $|Z_k(M,q)|$. 
The main result of this section is the following: 
\begin{proposition}\label{equalMOO}
Let $M_A$ and $M_B$ be SOL  torus bundles with the same 
absolute value MOO invariants $|Z_k(M,q)|$, for all $k$. 
Then,  either: 
\begin{equation}
{\rm Tr}(A)={\rm Tr}(B) 
\end{equation}
or else: 
\begin{equation}\label{sum4}
{\rm Tr}(A)+{\rm Tr}(B)=4
\end{equation}
Consequently those torus bundles having the same abelian 
Turaev-Viro invariants as $M_A$ fall into two 
commensurability classes. 
\end{proposition}
\begin{proof}
We have first the following explicit computation 
of the MOO invariants from \cite{MOO}:  
\begin{lemma}\label{MOO}
If $k$ is odd then we have: 
\begin{equation}
|Z_k(M,q)|=|H^1(M,\Z/k\Z)|^{1/2}
\end{equation}
If $k$ is even then: 
\begin{equation}
|Z_k(M,q)|=\left\{\begin{array}{ll}
|H^1(M,\Z/k\Z)|^{1/2}, & {\rm if} \: 
\alpha\cup\alpha\cup \alpha=0, {\rm for \, every}\, \alpha\in H^1(M,\Z/k\Z) \\ 
0, & {\rm otherwise} \\
\end{array}\right.
\end{equation}
\end{lemma}
\begin{proof}
See (\cite{MOO}, Thm.3.2).
\end{proof}

Further the cohomology of SOL torus bundles is given by: 
\begin{lemma}\label{kern}
If $M=M_A$ with $A\in SL(2,\Z)$ hyperbolic then 
\begin{equation}
H^1(M_A,\Z/k\Z)\cong \Z/k\Z\oplus \ker \nu_k(A^T-\mathbf 1) 
\end{equation}
where $A^T$ denotes the transposed of the matrix $A$. 
\end{lemma} 
\begin{proof}
By the Universal Coefficient Theorem 
$H^1(M_A,\Z/k\Z)\cong \Z/k\Z)={\rm Hom}(H_1(M_A),\Z/k\Z)$. 
Since $A$ is hyperbolic $H_1(M_A)=\Z\oplus {\rm Tors}(H_1(M_A))$.  
The torsion part can be computed by abelianizing $\Gamma_A$ and we find   
${\rm Tors}(H_1(M_A))=\Z^2/(A-{\bf 1})(\Z^2)$, which is a finite 
abelian group 
of order $|\det(A-{\mathbf 1})|=|{\rm Tr}(A)-2|$. 

Then ${\rm Hom}({\rm Tors}(H_1(M_A)), \Z/k\Z)$ is 
naturally identified with $\ker (A-{\mathbf 1})^*_k$, where \\
$(A-{\mathbf 1})^*_k: {\rm Hom}(\Z^2,\Z/k\Z)\to {\rm Hom}(\Z^2,\Z/k\Z)$ 
is the linear map given by 
$(A-{\mathbf 1})^*_k(f)=f\circ(A-{\mathbf 1})$, for  
$f\in {\rm Hom}(\Z^2,\Z/k\Z)$. 
We have a (non-canonical) isomorphism $(\Z/k\Z)^2\to 
{\rm Hom}(\Z^2,\Z/k\Z)$ which sends $(a,b)\in (\Z/k\Z)^2$ to the  
homomorphism $f_{a,b}$ satisfying 
$\left(f_{a,b}\left(\begin{array}{c}
1 \\
0
\end{array}\right),f_{a,b}\left(\begin{array}{c}
0 \\
1
\end{array}\right)\right)=(a,b) \in (\Z/k\Z)^2$. 
Then $f_{a,b}\in \ker (A-{\mathbf 1})^*_k$ if and only if 
$(a,b)\in \ker \nu_k(A^T-{\mathbf 1})$. This proves the claim. 
\end{proof}
Consider now two SOL manifolds $M_A$ and $M_B$ having the same 
absolute value MOO invariants. 
If the  MOO invariants as well as their generalizations from 
\cite{Del1} were  the same for the two manifolds 
then the result would be a simple consequence 
of the main theorem from \cite{DG}. In fact these invariants determine 
the linking pairing of the 3-manifold and in particular 
the torsion group ${\rm Tors}(H_1(M))$. 

The case where we know that the absolute value of the MOO invariants 
agree is  only slightly more complicated. First, when $k$ is odd 
Lemma \ref{MOO} and Lemma \ref{kern} imply that 
\begin{equation}\label{constr}
|\ker \nu_k(A^T-1)|=|\ker \nu_k(B^T-1)|
\end{equation}  

In order to compute the orders of the kernels above we have to recall 
some standard facts concerning the normal forms of integral matrices. 
Let $C:\Z^n\to \Z^n$  be a non-singular linear map 
$C:\Z^n\to \Z^n$. Then there exists 
a (unique) collection of positive integers $r_1,r_2,\ldots,r_n$, 
called the invariant factors of $C$ 
with $r_j$ dividing $r_{j+1}$ (when $j\leq n-1$) such that 
$C=V D W$, where $V,W\in GL(n,\Z)$ are invertible integral matrices 
and $D$ is diagonal with entries $r_1,r_2,\ldots,r_n$. 
Moreover $|\det(C)|=r_1r_2\cdots r_n$. 
This is the so-called Smith normal form (see \cite{New}, II.15).

This normal form is particularly useful if one seeks for counting the solutions 
of the congruences system 
$C(x)\equiv 0 ({\rm mod}\: k)$. 
By above this is equivalent to the system of congruences 
$r_jx_j\equiv 0 ({\rm mod}\: k)$, for $1\leq j\leq n$. 
Each congruence  above gives  ${\rm g.c.d.}(r_j, k)$ distinct 
solutions $x_j$ mod $k$, so 
that the total number of solutions of the system  is
$\prod_{j=1}^n {\rm g.c.d.}(r_j, k)$. 

Notice that the invariant factors for a 2-by-2 matrix  
$A=\left(\begin{array}{cc}
a & b \\
c & d \\
\end{array}\right)$ are simply 
$r_1(A)={\rm g.c.d.}(a,b,c,d)$ and $r_2(A)=\det(A)/r_1(A)$.

Now, for fixed $C$ and $k$ of the form $k=p^r$,  with prime $p$, if we choose 
$r$ large enough such that 
$r\geq m_{j,p}(C)$, where $r_j=p^{m_{j,p}(C)}s_j$, with 
${\rm g.c.d.}(p,s_j)=1$ then  
the previous discussion shows that  
\begin{equation}
|\ker \nu_k(C)|= {\rm g.c.d.}(|\det(C)|, k)
\end{equation}
We will apply this formula to $C=A-{\mathbf 1}$ and respectively 
$C=B-{\mathbf 1}$, where $k=p^r$ for odd prime $p$ and $r$ 
is chosen  large enough  such that 
\begin{equation}\label{condition}
r\geq \max(m_{j,p}(A-{\mathbf 1}), m_{j,p}(B-{\mathbf 1})) 
\end{equation}
Then the relations above imply that 
\begin{equation}
{\rm g.c.d.}({\rm Tr}(A)-2, p^r)=
{\rm g.c.d.}({\rm Tr}(B)-2, p^r)
\end{equation}
for every odd prime $p$ and $r$ large enough. Therefore 
the numbers $|{\rm Tr}(A)-2|$ and $|{\rm Tr}(B)-2|$ have the same 
odd divisors. 

Let now call  the even number $k$ to be {\em good for} $M$ if we have 
$\alpha\cup\alpha\cup \alpha=0$, for every 
$\alpha\in H^1(M,\Z/k\Z)$. Lemma \ref{MOO} shows that 
$k$ is good if and only if $|Z_k(M,q)|\neq 0$. 
On the other hand in (\cite{MOO}, Cor. 5.3) one founds the following 
explicit criterion.  The number $k$ is not good for $M$, i.e. 
$Z_k(M,q)= 0$, if and only if  there exists 
$x\in {\rm Tors}(H_1(M))$ of order $2^m$ such that 
$L_M(x,x)=\frac{c}{2^m}$,  where $k=2^mb$, with odd $b$,  $L_M$ denotes 
the linking pairing 
$L_M: {\rm Tors}(H_1(M))\times {\rm Tors}(H_1(M))\to \Q/\Z$, 
and $c$ is odd.

Now, if $M_A$ and $M_B$ have the same  absolute value of MOO 
invariants then $k$ is good for $M_A$ if and only if $k$ is good for $M_B$. 
On the other hand, we know that $|{\rm Tr}(A)-2|=2^{m_{A}}s$ and 
$|{\rm Tr}(A)-2|=2^{m_B}s$, with odd $s$. 
Observe that any $k$ of the form $k=2^{r}$, with $r \geq m_A+1$ is good for 
$M_A$ since the torsion ${\rm Tors}(H_1(M_A))$ has no 
elements of order $2^r$. In particular if $r\geq \max(m_A,m_B)+1$ 
then $2^r$ is good for both $M_A$ and $M_B$.

Choose now  $p=2$, $k=2^r$ with $r$ large enough as in (\ref{condition}) 
and such that $2^r$ is good for both $M_A$ and $M_B$.  
Then the equality of MOO invariants of $M_A$ and $M_B$ implies 
\begin{equation}
|\ker \nu_k(A^T-{\mathbf 1})|=|\ker \nu_k(B^T-{\mathbf 1})|
\end{equation}
which, by above, is equivalent to the following:  
\begin{equation}
{\rm g.c.d.}({\rm Tr}(A)-2, 2^r)=
{\rm g.c.d.}({\rm Tr}(B)-2, 2^r)
\end{equation}
Thus $m_A=m_B$ and this completes the proof of the fact that:  
\begin{equation}
|{\rm Tr}(A)-2|= 
|{\rm Tr}(B)-2|
\end{equation}
If ${\rm Tr}(A)={\rm Tr}(B)$ then 
$A$ and $B$ have the same trace and the same 
determinant and thus the equation 
$XA=BX$ has solutions in $GL(2,\Q)$, so that $M_A$ and $M_B$ 
are (strongly) commensurable. 

In fact, recall that Barbot and later Bridson and Gersten  
(\cite{Bar,BG}) proved the following: 
\begin{lemma}\label{commensur}
The groups $\Gamma_A$ and 
$\Gamma_B$ are commensurable if and only if the quotient 
of their discriminants $D_A/D_B$ is the square of an rational.  
Here the discriminant of $A$ is $D_A={\rm Tr}(A)^2-4\det(A)$. Moreover, this is 
equivalent to the fact that $A^p$ and $B^q$ are conjugate within $GL(2,\Q)$, 
for some $p,q\in\Z$. 
\end{lemma}

Further if ${\rm Tr}(A)+{\rm Tr}(B)=4$ we have again 
only one (strong) commensurability class allowed for $B$. 
Thus the torus bundles as in the statement of the Proposition 
fall into two commensurability classes. 
\end{proof}

We will give now several examples to show that all abelian 
invariants (of Reshetikhin-Turaev type, not only their absolute values) 
fail to distinguish the two distinct commensurability classes above. 

\begin{proposition}
Set 
\begin{equation}
A=\left(\begin{array}{cc}
1 & n \\
1 & n+1 \\
\end{array}
\right), \,\, B=\left(\begin{array}{cc}
1-2n & n \\
-1-2n & n+1 \\
\end{array}
\right) \:\:\; n\in \Z_+
\end{equation}
The manifolds $M_A$ and $M_B$ have the same quantum abelian 
invariants although ${\rm Tr}(A)+{\rm Tr}(B)=4$ and ${\rm Tr}(A)\neq {\rm Tr}(B)$, if $n\geq 1$ and 
$n\neq 4$. 
In particular the trace is not detected by the quantum abelian invariants 
of torus bundles. Moreover, if $\frac{n+4}{n-4}\not\in \Q^2$ then 
$M_A$ and $M_B$ (equivalently $\Gamma_A$ and $\Gamma_B$) are not commensurable. 
\end{proposition}
\begin{proof}
The quantum abelian invariants from \cite{Del1}  are identical 
for two manifolds if and only if their 
first Betti numbers agree and their linking pairings are isomorphic 
(see \cite{DG}). 

 Let $T$ be the torus fiber of $M_A$. The we have the 
exact sequence: 
\begin{equation}
H_2(M_A)\to H_1(T)\stackrel{A-\mathbf 1}{\to} H_1(T)\to H_1(M_A)\to \Z
\end{equation} 
Therefore $H_1(M_A)=\Z\oplus {\rm Tors}(H_1(M_A))$, where the torsion 
${\rm Tors}(H_1(M_A))$ is the image of $H_1(T)$ into $H_1(M_A)$. 

The linking pairing $L_A:{\rm {\rm Tors}}(H_1(M_A))\times {\rm {\rm Tors}}(H_1(M_A))\to \Q/\Z$ 
is defined as follows. For every $\xi \in {\rm {\rm Tors}}(H_1(M_A))$ we choose 
a lift of it as an element in $H_2(M_A;\Q/\Z)$, namely an element $\hat{\xi}$
whose image by the boundary connecting homomorphism 
$\delta_*:H_2(M_A;\Q/\Z)\to H_1(M,\Z)$ is exactly $\xi$. 
Here the connecting homomorphism comes from the 
long exact sequence associated to the coefficients exact sequence: 
\begin{equation}
\to H_2(M_A;\Q)\to H_2(M_A;\Q/\Z)\to H_1(M,\Z)\to H_1(M,\Q)\to 
\end{equation}
We take then $L_A([\eta],[\xi])=\eta \cdot \hat{\xi}\in \Q/\Z$ 
where the intersection product is the one 
$H_1(M,\Z)\times H_2(M_A;\Q/\Z)\to \Q/\Z$. 

If we have a 1-cycle $\xi$ representing the class $[\xi]\in H_1(T^2)$ 
then its product with $[0,1]$ yields 
a 2-chain whose boundary is $(A-\mathbf 1)\xi$. This implies 
that the linking pairing of $M_A$ is given by: 
\begin{equation}
L_A([\eta],[\xi]) =\omega((A-\mathbf 1)^{-1}(\eta),\xi)\in \Q/\Z
\end{equation} 
where $\eta,\xi\in H_1(T)\cong \Z^2$ are representing (torsion) classes in 
$\Z^2/(A-\mathbf 1)(\Z^2)\subset H_1(M_A)$ and 
$\omega$ is the usual (symplectic) intersection form on $H_1(T)$, namely 
\begin{equation}
\omega((v_1,v_2),(w_1,w_2))=v_1w_2-v_2w_1
\end{equation}  
The torsion group of 1-homologies of $M_A$ and $M_B$ are both 
cyclic groups of order $|{\rm Tr}(A)-2|$ since the first invariant 
factors for the integral matrices $A-\mathbf 1$ and $B-\mathbf 1$ are 
both equal to 1. Thus the torsion homology groups are isomorphic.   
Now we can verify that 
$(A-\mathbf 1)^{-1}-(B-\mathbf 1)^{-1}$ 
is the integral matrix 
$\left(\begin{array}{cc}
2 & -2 \\
2 & -2 \\
\end{array}
\right)$ such that the linking pairings of $M_A$ and $M_B$ are isomorphic. 
If $n\geq 1$ and $n\neq 4$ then these torus bundles are SOL manifolds.
Their Betti numbers coincide as all SOL manifolds have their 
first Betti number equal to 1.

The statement concerning the commensurability is a consequence of the 
commensurability criterion for the polycyclic groups from 
Lemma \ref{commensur} saying that $\Gamma_A$ and 
$\Gamma_B$ are commensurable if and only if ${\mathcal D}_A/\mathcal D_B\in \Q^2$. 
\end{proof}

\begin{remark}
If $n=4$ then $M_A$ is a SOL torus bundle 
but $M_B$ is a NIL manifold. Although their linking pairings are isomorphic 
their first Betti numbers are different, as the Nil manifold 
has Betti number 2. Another pairs with the same property are 
\begin{equation}
A=\left(\begin{array}{cc}
3 & 2 \\
4 & 3 \\
\end{array}
\right), \,\, B=\left(\begin{array}{cc}
3 & 2 \\
-8 & -5 \\
\end{array}
\right) 
\end{equation} 
\end{remark}

\subsection{$SU(2)$-invariants and the metaplectic representations}

Denote by $\rho_{SU(2),k}$ (and respectively $\rho_{U(1),k}$) 
the $SL(2,\Z)$ representation associated to the 
modular tensor category constructed out of $SU(2)$ (and respectively 
$U(1)$ or $\Z/k\Z$) in level $k$ (see \cite{Tu}). 
It should be noticed that the parameters $\lambda_{SU(2),k}, \zeta_{SU(2),k}$ 
do not agree with $\lambda_{U(1),k}, \zeta_{U(1),k}$. 
For instance $\zeta_{U(1),k}=\exp\left(\frac{\pi i}{4}\right)$ is independent on $k$. 
The choice of the rank and anomaly will be irrelevant in the arguments below.

Recall first that both representations 
$\rho_{SU(2),k}$ and $\rho_{U(1),k}$ factor through the finite congruence group 
$SL(2,\Z/k\Z)$. 

Now explicit formulas for the values of $SU(2)$ quantum invariants 
of torus bundles were obtained in \cite{J} by Jeffrey. Nevertheless 
it seems difficult to extract explicit topological information out of them. 

The key point in our computation is the existence of simple formulas for the 
characters of the $SU(2)$ quantum representations: 

\begin{proposition}\label{SUchar}
We have 
\begin{equation}
2 {\rm Tr}(\rho_{SU(2),k}(A))=
{\rm Tr}(\rho_{U(1),k}(A))-
{\rm Tr}(\rho_{U(1),k}(-A))
\end{equation}
\end{proposition}
\begin{proof}
The finite symplectic groups $Sp(2g,\Z/k\Z)$ are endowed with 
(projective) representations into some complex 
vector space $V_k$, which are known under the name 
of Segal-Shale-Weil metaplectic representations. Although these 
were classically constructed only for prime $k$ there exist 
now several constructions valid for every $k$.  
In \cite{F1,F2,Go} one constructed 
such representations for every even $k$ (and for a  congruence 
quotient of the Theta group $\Gamma[2]$ when $k$ is odd) 
in any dimension $g$ using level $k$ theta functions. 
The monodromy representations 
from \cite{MOO} agree with the previous constructions and work for every 
odd $k$ as well. 
Later in \cite{FHKMN}  one described a direct construction of the 
metaplectic $SL(2,\Z/k\Z)$ representations which were further 
generalized in \cite{KN} to higher dimensions.

The following seems to be widely known among experts: 
\begin{lemma}
The  $SL(2,\Z)$ quantum representations $\rho_{U(1),k}$
are lifts of the projective metaplectic representations.  
\end{lemma}

The theta functions construction was generalized in 
\cite{F3} to quantizations of multidimensional tori endowed with 
Coxeter group actions. This leads to finite symplectic group 
representations depending on a semisimple Lie group $G$ or, equivalently 
on a Coxeter group $W$ (corresponding to the Weyl group of $G$). 
It was already noticed in \cite{F3} 
that the $SL(2,\Z/k\Z)$ representations associated to 
$W=\Z/2\Z$ coincide (projectively) with $\rho_{SU(2),k}$. 

\begin{lemma}
Let $\tau=\left(\begin{array}{ll} 
-1 & 0 \\
0 & -1
\end{array}
\right)\in SL(2,\Z)$. The space $V_k$ splits into eigenspaces 
for the metaplectic action of $\tau$ as $V_k=V_k^{+}\oplus V_k^{-}$, where 
\begin{equation}
V_k^{\pm}=\{x\in V_k; \:  \rho_{U(1),k}(\tau)(x)=\pm x\}
\end{equation}  
Then the representation $\rho_{SU(2),k}$ of $SL(2,\Z)$ is 
isomorphic to the restriction $\rho_{U(1),k}|_{V_k^{-}}$ of the metaplectic 
representation  to the invariant sub-module ${V_k^{-}}$.  
\end{lemma}
\begin{proof}
This was made so by the explicit construction in \cite{F3}. 
The result was also formulated explicitly in (\cite{FK}, section 5) 
for prime $k$, but the same argument is valid for all $k$ 
when comparing with the formulas in \cite{FHKMN}. 
A more precise result was given by
Larsen and Wang in \cite{LW} and independently by Gilmer in 
(\cite{Gil}, Thm.5.2). 
\end{proof}

The two lemmas above prove the claim, since the characters of the factors 
$V^{\pm}$ are precisely the $\pm$-invariant part of the character of $V_k$. 
\end{proof}

Recall now from Proposition \ref{invar} and equation (\ref{RTinvar}) 
that the  Reshetikhin -Turaev quantum invariants 
of the torus bundle $M_A$ are suitable 
multiples of the corresponding characters, as follows:  
\begin{equation}\label{RT}
RT_{SU(2),k}(M_A)=\zeta_{SU(2),k}^{-3\varphi(A)} 
{\rm Tr}(\rho_{SU(2),k}(A), \:\;  
RT_{U(1),k}(M_A)=\zeta_{U(1),k}^{-3\varphi(A)} 
{\rm Tr}(\rho_{U(1),k}(A)
\end{equation} 
 
The Turaev-Viro abelian invariant is known to be the same as the 
absolute value of the MOO invariant (up to a scalar) and this can be extended 
as follows:  

\begin{lemma}
For any oriented 3-manifolds we have: 
\begin{equation}
RT_{U(1),k}(M)=k^{-1/2}Z_k(M,q)
\end{equation}
\end{lemma}
\begin{proof}
We know that $TV_{U(1),k}(M_A)=k^{-1/2}|Z_k(M,q)|$ and the associated 
projective representations are isomorphic (see \cite{MOO,F1,F2,Go}). 
The anomalies are the same and thus the associated Reshetikhin-Turaev 
invariants agree. 

Another proof is given in (\cite{Del1}, Appendix A) where one uses 
the modular tensor category from (\cite{Tu}, p.29). 
\end{proof}

Assume now that $RT_{SU(2),k}(M_A)=RT_{SU(2),k}(M_B)$ and 
$RT_{U(1),k}(M_A)=RT_{U(1),k}(M_B)$. 
Then Proposition \ref{SUchar} and relations (\ref{RT}) imply that 
$RT_{U(1),k}(M_{\tau A})=RT_{U(1),k}(M_{\tau B})$. In particular, 
applying the result of Proposition \ref{equalMOO} we obtain that either 
${\rm Tr}(A)={\rm Tr}(B)$ or else 
${\rm Tr}(-A)+{\rm Tr}(-B)=4$. The only possibility is that 
${\rm Tr}(A)={\rm Tr}(B)$.

The case when the Turaev-Viro invariants of the two 
manifolds agree is only slightly more complicated. 
The key point is that 
Proposition \ref{SUchar} leads  to a closed formula for the 
$SU(2)$ quantum invariants of torus bundles. We restrict, 
for the sake simplicity, to the case of  
Turaev-Viro invariants, which are central in our argument.

\begin{proposition}\label{closedform}
Let $A\in SL(2,\Z)$ and $k$ be large enough such that whenever  
$p^m$, with prime $p$ and $m\geq 1$, divides some invariant factors   
of $A-{\mathbf 1}$ or  $A+\mathbf 1$ then it also divides $k$.  
Then the $SU(2)$-Turaev-Viro invariant of $M_A$ is given by 
\begin{eqnarray}\label{TVS}
|{\rm Tr}(\rho_{SU(2),k}(A)|^2&=& TV_{SU(2),k}(M_A)=\nonumber \\
&=&\left(\sqrt{{\rm g.c.d.}({\rm Tr}(A)-2, k)}-
\exp\left(\frac{\pi i}{4}\left(f_{k}(M_A)\right)\right) 
\sqrt{{\rm g.c.d.}({\rm Tr}(A)+2,k)}\right)^2 
\end{eqnarray}
where $f_{k}(M_A)=\phi_k(M_{\tau A})-\phi_k(M_A)$, $\tau A=-A$ and 
$\phi_k(M_A)\in \Z/8\Z$ is the function introduced in 
(\cite{MOO}, section 4).  
\end{proposition}
\begin{proof}
We need first the following: 
\begin{lemma}
If $A$ is hyperbolic then $\varphi(A)=\varphi(\tau A)$. 
\end{lemma}
\begin{proof}
By definition $\Phi_R(A)=\Phi_R(-A)$ since the Rademacher function 
is defined on $PSL(2,\Z)$. Further, by (\ref{difer})  
the function $\varphi(A)-\Phi_R(A)$ 
is  equal to ${\rm sgn}(\gamma(\alpha+\delta-2))$  
when $A=\left(\begin{array}{cc}
\alpha & \beta \\
\gamma & \delta \\
\end{array}\right)$ and so it also 
satisfies $\varphi(A)-\Phi_R(A)=\varphi(-A)-\Phi_R(-A)$ 
when $A$ is hyperbolic, by direct inspection.  
\end{proof}
The last lemma implies that 
\begin{equation}
Z_k(M_A,q)-Z_k(M_{\tau A},q)=
\zeta_{U(1),k}^{-3\varphi(A)} 
({\rm Tr}(\rho_{U(1),k}(A)-{\rm Tr}(\rho_{U(1),k}(\tau A))
\end{equation}
We have the following: 
\begin{lemma}\label{even}
If $A$ is hyperbolic and $k$ is good for $M_A$ and sufficiently large 
then 
\begin{equation}
{\rm Tr}(\rho_{U(1),k}(A) = \exp\left(\frac{\pi i}{4}\left(\varphi(A)+\phi_k(M_A)\right)\right) |\ker \nu_k(A-{\mathbf 1})|
\end{equation}
where $\phi_k(M_A)\in \Z/8\Z$ is the function introduced in 
(\cite{MOO}, section 4).  
\end{lemma}
\begin{proof}
The MOO invariant was computed in (\cite{MOO}, Thm.4.5) for those $k$ 
for which the invariant is non-zero, as being 
\begin{equation}
Z_k(M_A,q)=\exp\left(\frac{\pi i}{4}\left(\varphi(A)+\phi_k(M_A)\right)\right) 
|H^1(M_A,\Z/k\Z)|
\end{equation}
Since $\zeta_{U(1),k}^3=\exp\left(\frac{\pi i}{4}\right)$ we 
obtain: 
\begin{equation}
{\rm Tr}(\rho_{U(1),k}(A) = \zeta^{3\varphi(A)}RT_{U(1),k}(M_{\tau A})=
\exp\left(\frac{\pi i}{4}\left(\varphi(A)+\phi_k(M_A)\right)\right) 
k^{-1/2}|H^1(M_A,\Z/k\Z)|^{1/2}
\end{equation}
which implies the claim. 
\end{proof}
Now, if $A$ is hyperbolic  and $k$ is large enough 
then use Lemma \ref{even} to derive: 
\begin{eqnarray}\label{TVSU}
|{\rm Tr}(\rho_{SU(2),k}(A)|^2&=& TV_{SU(2),k}(M_A)= |RT_{SU(2),k}(M_A)|^2=
k^{-1}|Z_k(M_A,q)-Z_k(M_{\tau A},q)|^2= \nonumber \\
&=&\left||\ker \nu_k(A^T-{\mathbf 1})|^{1/2}-
\exp\left(\frac{\pi i}{4}\left(\phi_k(M_{\tau A})-\phi_k(M_A)\right)\right) 
|\ker \nu_k(A+{\mathbf 1})|^{1/2}\right|^2 
\end{eqnarray}
Then the closed formula (\ref{TVS}) follows. 
\end{proof}

\subsection{End of the proof of Theorem \ref{comm}}
It remains to prove the following: 
\begin{proposition}\label{retrievetrace}
If the SOL torus bundles 
$M_A$ and $M_B$ have the same abelian and $SU(2)$ Turaev-Viro 
invariants then ${\rm Tr}(A)={\rm Tr}(B)$. 
\end{proposition}
\begin{proof}
We have to recall (see e.g. \cite{Tu}, section VI) that 
the modular tensor category ${\mathcal C}_{SU(2)}$ which is 
leading to the $SU(2)$ invariants is defined only when the level 
is of the form $4n$, with $n\geq 3$.

We assume that 
$a={\rm Tr}(A)\neq {\rm Tr}(B)$. 
According to Proposition \ref{equalMOO} we must have 
${\rm Tr}(B)=4-a$. Let 
$k$ be large enough in order to be  
good for $M_A$ and $M_B$ and also to verify (\ref{condition}).
We put ${\rm g.c.d.}(a-2, k)=u$, 
 ${\rm g.c.d.}(a+2, k)=v$ and 
 ${\rm g.c.d.}(a-6, k)=w$. 
Then (\ref{TVSU}) implies that 
\begin{equation}
-2\cos\left(\frac{\pi i}{4}f_A\right)
\sqrt{uv} + v=
-2\cos\left(\frac{\pi i}{4}f_A\right)
\sqrt{uw} + w
\end{equation}
where $f_A=\phi_k(M_{\tau A})-\phi_k(M_A)$. 

This is equivalent to the equation:  
\begin{equation}\label{eqtr}
(\sqrt{v}-\sqrt{w})\left(\sqrt{v}+\sqrt{w}-2\cos\left(\frac{\pi i}{4}f_A\right)
\sqrt{u}\right) = 0
\end{equation}

\begin{lemma}
The prime divisors of $a-6$ are the same as the prime divisors of $a+2$.
\end{lemma}
\begin{proof}
Suppose that there exists some odd $p$ which divides 
$a-6$ but not $a+2$. We write 
$a-6=2^sp^rc$, with $c$ odd and coprime with $p$, $r\geq 1$.

Assume first $s\geq 3$. We chose  $k$ of the form $k=2^mp^m$ (with 
$m$ large with respect to $r$ and $s$). 
Then $w=2^sp^r$, $u={\rm g.c.d.}(4(2^{s-2}p^rc+1),2^mp^m)=4$, 
$v={\rm g.c.d.}(8(2^{s-3}p^rc+1),2^mp^m)=2^t$, where $t\geq 3$. Actually 
we have $t=3$ if $s\geq 4$. 
Then equation (\ref{eqtr}) implies that:
\begin{equation}
\sqrt{2^sp^r}+\sqrt{2^t}=4\cos\left(\frac{\pi i}{4}f_A\right)  
\end{equation}
Since $4\cos\left(\frac{\pi i}{4}f_A\right)\in\{0,\pm2\sqrt{2}, 
\pm 4\}$ this equation is impossible for any odd prime $p$.

Consider now $s=1$. We choose again $k$ of the form $k=2^mp^m$, with 
$m$ large with respect to $r$. 
Then $w=2p^r$, $u={\rm g.c.d.}(2(p^rc+2),2^mp^m)=2$, 
$v={\rm g.c.d.}(2(p^rc+4),2^mp^m)=2$, so that 
equation (\ref{eqtr}) implies that:
\begin{equation}
\sqrt{2^sp^r}+\sqrt{2}=2\cos\left(\frac{\pi i}{4}f_A\right)\sqrt{2}  
\end{equation}
Its only integral solution is $p=1$ which is not convenient. 

Let now $s=0$. Then chose again $k$ of the form $k=2^mp^m$ (with 
$m$ large with respect to $r$). We find that 
$w=p^r$, $u={\rm g.c.d.}(p^rc+4,2^mp^m)=1$, 
$v={\rm g.c.d.}(p^rc+8,2^mp^m)=1$, so that 
equation (\ref{eqtr}) above yields:
\begin{equation}
\sqrt{p^{r}}+1=2\cos\left(\frac{\pi i}{4}f_A\right)
\end{equation}
The only integral solution is again $p=1$. 

Eventually, let us consider the case when $s=2$. We write 
$p^rc+1=2^ud$, with odd $d$. Chose now $k=2^mp^mc^m$, for some 
large enough $m$ so that $k$ is good for $M_A$ and $M_B$ and 
verifies  (\ref{condition}). Then 
$w=4p^rc$, $u={\rm g.c.d.}(4(p^rc+1),2^mp^mc^m)=2^{u+2}$, 
$v={\rm g.c.d.}(4(p^rc+2),2^mp^mc^m)=4$. In this case 
equation (\ref{eqtr}) gives us: 
\begin{equation}
\sqrt{p^{r}c}+1=2\cos\left(\frac{\pi i}{4}f_A\right)\sqrt{2^u}
\end{equation}
Suppose that $\cos\left(\frac{\pi i}{4}f_A\right)=1$ so that 
we have to find integral solutions of: 
 \begin{equation}
1+\sqrt{2^ud-1}=\sqrt{2^{u+2}}
\end{equation}
If $d\geq 5$ then for every $u\geq 1$ we have: 
\begin{equation}
1+\sqrt{2^ud-1}\geq 1+\sqrt{5\cdot 2^u-1} > 2\sqrt{2^u}
\end{equation}
If $d=3$ then the previous equation is equivalent to 
\begin{equation}
1+\sqrt{3\cdot 2^u-1}= 2\sqrt{2^u}
\end{equation}
By taking the square and collecting together the terms 
we derive that $2^{2u-2}=3\cdot 2^u-1$. This is impossible 
when $u\geq 1$ because of modulo 2 considerations. 
If $d=1$ then 
\begin{equation}
1+\sqrt{2^u-1} < 2\sqrt{2^u}
\end{equation}
The only possibility left is that  
$2\cos\left(\frac{\pi i}{4}f_A\right)=\sqrt{2}$ so that the 
equation reads: 
\begin{equation}
1+\sqrt{2^ud-1}=\sqrt{2^{u+2}}
\end{equation}
If $d\geq 3$, as $u\geq 1$, we have: 
\begin{equation}
1+\sqrt{2^ud-1}\geq 1+\sqrt{3\cdot 2^u-1} > 2\sqrt{2^u}
\end{equation}
If $d=1$ then the equation reads: 
\begin{equation}
1+\sqrt{2^u-1} = 2\sqrt{2^u}
\end{equation}
Squaring both sides and collecting the terms we obtain 
$2^{2u-2}=2^u-1$, which is impossible by mod 2 considerations. 
This proves that any odd prime dividing $a-6$ also divides $a+2$. 
A similar proof shows that 
conversely, if an odd prime $p$ divides $a+2$ then $p$ divides $a-6$.
This proves the Lemma. 
\end{proof}

Thus the prime divisors of $a-6$ and $a+2$ are the same and this implies 
that they divide their difference, so actually the only prime divisor 
of these two numbers is 2. Thus 
$a-6=\pm 2^m$ and $a+2=\pm 2^n$, for some integers $m,n$. 
This is impossible when $m\geq 5$ since it implies that 
$8(\pm 2^{m-3}+1)=\pm 2^n$, but  $\pm 2^{m-3}+1$ is a non-trivial 
odd number.  Inspecting the remaining cases when $0\leq m\leq 4$ 
leads us to the following solutions 
$a=2$, $a=14$ and $a=-10$. The first is not convenient since 
$A$ was supposed hyperbolic. 
The other ones do not satisfy  the constraint (\ref{eqtr}). 
This contradiction shows that the only possibility is that 
${\rm Tr}(A)={\rm Tr}(B)$, as claimed. 
\end{proof}

\subsection{Ideal class groups and proofs of 
Corollaries  \ref{class} and \ref{plat}}
We want to prove that the set of those 
$M_B$ having the same abelian  and $SU(2)$ 
Turaev-Viro invariants as $M_A$ is finite, and it can be identified 
with a subset of  a quotient of 
$\mathcal I(M_A)$ by the 
involution $\iota$ which acts as $X\to X^{-1}$ on matrices with 
given trace.

Let $\alpha$ be a root of $x^2-{\rm Tr}(A) x +1=0$, where 
$|{\rm Tr}(A)|\neq 2$. 
A construction due to Latimer, MacDuffee 
and Taussky-Todd (see \cite{TT} and \cite{New}, III.16 for details) 
establishes a one to one correspondence 
between the ideal class group  $\mathcal I(M_A)$ of the order 
$\Z[\alpha]$ and the classes of matrices $C\in SL(2,\Z)$ 
with trace ${\rm Tr}(C)={\rm Tr}(A)$, considered 
up to conjugacy in $GL(2,\Z)$.  

The order $\Z[\alpha]$ is sometimes (though not always) 
the ring of integers of a real quadratic field. Specifically, set 
$D_A={\rm Tr}(A)^2-4$, for odd ${\rm Tr}(A)$ and 
$D_A=\frac{1}{4}{\rm Tr}(A)^2-1$, for even ${\rm Tr}(A)$, respectively. 
If $D_A$ is squarefree, then $\Z[\alpha]$ is the 
ring of integers ${\mathcal O}_{\sqrt{D_A}}$ of the 
real quadratic field $\Q(\sqrt{D_A})$. 

For any $SL(2,\Z)$ matrix $C$ having trace ${\rm Tr}(A)$ 
one defines an ideal of $\Z[\alpha]$ as follows. 
Consider an eigenvector $(u_1,u_2)$ of $C$ 
associated to the eigenvalue $\alpha$, which could be chosen to lie 
within $\Z[\alpha]\times \Z[\alpha]$. 
Therefore $\{u_1,u_2\}$ form the basis of an ideal $I(C)\subset 
\Z[\alpha]$. Conversely, the choice of a basis 
of an ideal $I\subset \Z[\alpha]$ determines a matrix $C(I)\in SL(2,\Z)$ 
corresponding to the multiplication by $\alpha$. 
This matrix is uniquely determined by $I$, 
up to conjugacy in $GL(2,\Z)$. 

In the ideal class group  $\mathcal I(M_A)$ of $\Z[\alpha]$ 
two ideals $I$ and $J$ are identified if there 
exist nonzero elements $v,w\in \Z[\alpha]$ 
such that $vI=wJ$. Further, if $B=UCU^{-1}$, with $U\in GL(2,\Z)$, then 
the ideals $I(B)$ and $I(C)$ are equivalent.
Therefore the class of $I(C)$ is well-defined in 
$\mathcal I(M_A)$, independently on the representative $C$ in its 
conjugacy class.

Now recall that two torus bundles manifolds $M_A$ and $M_B$ 
are homeomorphic if and only if their fundamental groups are isomorphic, 
since they are aspherical. According to Proposition 
\ref{classif} this corresponds to the fact that $A$ is conjugate to 
$B$ or to $B^{-1}$ within $GL(2,\Z)$. 
If we take into account the involution $B\to B^{-1}$ 
we obtain the first claim of the Corollary \ref{class}.  
Eventually Dedekind's Theorem states the 
finiteness of the ideal class group and it permits to conclude.  

Although the statement of Proposition \ref{classif} was only stated for 
hyperbolic matrices $A$ and $B$ this extends naturally to all 
matrices from $SL(2,\Z)$.

Eventually, stronger results dues to Platonov 
 and Rapinchuk (see \cite{Platonov,Rapin}, 
\cite[ section 8.8.5]{PR}) 
show that the number of classes in an arithmetic group belonging 
to the same $G$-genus (where $G$ is a connected linear algebraic group 
defined over $\Q$) is finite and unbounded. In particular, the 
number of classes in $\mathcal X^{TV}(M)$  is unbounded. 
This settles Corollary \ref{plat}.

\begin{remark}
One should notice that there exist classes of matrices $B$ in 
$I(M_A)/\iota$ such that $B$ and $A$ are not 
conjugate in every congruence quotient.
If ${\mathcal I}^{A; loc}(M_A)\subset \mathcal I(M_A)$ 
is the set of of conjugacy classes of matrices $B$ which are conjugate 
in every congruence quotient to a given $A$ 
it would be interesting to know the behavior of 
the  $|{\mathcal I}^{A; loc}(M_A)|$ when $D_A$ goes to infinity and 
also of the relative density of this subset in $\mathcal I(M_A)$. 
\end{remark}

\subsection{Proof of Corollary \ref{profinite}}
For every finite group $F$ there is associated a modular category 
whose associated invariants are the so-called  
Dijkgraaf-Witten invariants (see e.g. \cite{Tu}). The simplest 
of them is the untwisted Dijkgraaf-Witten
invariant $RT_F$ given by the following explicit counting formula 
in terms of the fundamental group of the closed 3-manifold $M$ (according to 
\cite{Tu} or \cite{FQ}, (5.14)): 
\begin{equation}
RT_F(M)=\frac{1}{|F|}|{\rm Hom}(\pi_1(M), F)|, 
\end{equation}

We have now the following easy lemma: 

\begin{lemma}\label{finitequot}
Let $\Gamma_1$ and $\Gamma_2$ be finitely generated groups such that: 
\begin{equation}
|{\rm Hom}(\Gamma_1, F)|= |{\rm Hom}(\Gamma_2, F)|
\end{equation}
holds for any finite group $F$. Then the sets of finite quotients 
of $\Gamma_1$ and $\Gamma_2$ respectively, coincide. 
\end{lemma}

Eventually, recall that the profinite completions 
of two finitely generated groups are isomorphic as 
topological groups if and only if the sets of their finite 
quotients are the same (see \cite{DFPR}). 
However, two profinite completions are isomorphic as topological groups 
if and only if they are isomorphic as discrete groups, because 
finite index subgroups in profinite groups are open, according to a fundamental 
result of Nikolov and Segal (\cite{NiS} and the discussion in \cite{DFPR}). 
This settles Corollary \ref{profinite}. 

\begin{proof}[Proof of  Lemma \ref{finitequot}]
Let ${\rm Hom}^{\rm surj}(\Gamma, F)$ denotes the set of 
surjective homomorphisms between the groups $\Gamma$ and $F$. 
We claim first that, under the assumptions of the lemma, we have for any finite group $F$ the equality: 
\begin{equation}
|{\rm Hom}^{\rm surj}(\Gamma_1, F)|= |{\rm Hom}^{\rm surj}(\Gamma_2, F)|
\end{equation}
Otherwise, pick up some $F$ for which the claim above is false 
and such that $F$ is a minimal group, 
with respect to the inclusion, with this property. 
Then $F$ is nontrivial and 
\begin{equation}
|{\rm Hom}^{\rm surj}(\Gamma_1, F)|\neq |{\rm Hom}^{\rm surj}(\Gamma_2, G)|
\end{equation}
 By the induction hypothesis we have: 
\begin{equation}
|{\rm Hom}^{\rm surj}(\Gamma_1, G)|= |{\rm Hom}^{\rm surj}(\Gamma_2, G)|
\end{equation}
for any subgroup $G\subset F$ such that $G\neq F$. However, we also have: 
\begin{equation}
|{\rm Hom}(\Gamma_i, F)|= \sum_{G\subset F}|{\rm Hom}^{\rm surj}(\Gamma_i, G)|
\end{equation} 
The inequality above implies then 
\begin{equation}
|{\rm Hom}(\Gamma_1, F)|\neq |{\rm Hom}(\Gamma_2, G)|
\end{equation}
contradicting our assumptions. 
This proves the claim. 

Eventually, observe that $F$ is a finite quotient of the group 
$\Gamma_i$ if and only if $|{\rm Hom}^{\rm surj}(\Gamma_i, F)|\neq 0$.  
Then the claim above implies that the set of finite quotients 
of the groups $\Gamma_i$ should coincide. 
\end{proof} 
\subsection{Proof of Proposition \ref{finiteclass}}
Let $G$ be the fundamental group of a closed orientable irreducible SOL manifold $M$. 
Then $G$ is solvable and according to a result of Evans and Moser 
(see \cite{EM}, Theorem 5.2) $G$ is polycyclic. 
 
Consider the fundamental group $H$ of a closed 3-manifold 
whose class is in $\mathcal X^{TV}(M)$.   
According to Lemma \ref{finitequot} the finite quotients of $H$ coincide 
with the finite quotients of $G$. 
Moreover, by classical results of Hempel and Perelman's solution to the 
geometrization conjecture the 3-manifold groups are residually finite. 
Sabbagh and Wilson have proved in \cite{SW} that any residually finite group $H$ 
having the same quotients as a polycyclic group is also polycyclic. 
In particular $H$ is polycyclic. Now the finiteness statement is a consequence 
of a deep theorem of Grunewald, Pickel and Segal (see \cite{GPS}) which states 
that the number of polycyclic groups with the same profinite completion is finite.

\section{Comments}
\subsection{Higher genus}
A direct extension of these results to higher genus surface bundles 
does not seem to work. In the case of the closed torus the 
kernel of all modular representations  of level $k$ is a 
congruence subgroup of level $k$ and hence strictly larger than  
the normal subgroup generated by the $k$-th powers of Dehn twists.  
In higher genus one expects the kernel of 
$SU(2)$ quantum representation to be precisely 
the normal subgroup generated by the $k$-th powers of Dehn twists.
 
The first case to analyze is the the mapping class group of 
the 1-punctured torus ${\mathcal M}_{1}^1$ (isomorphic to $SL(2,\Z)$). 
Its quantum representations are known 
not being always congruence anymore. Moreover, the kernel of the quantum 
$SU(2)$-representations (where the puncture is 
colored with every possible color) is now the 
subgroup ${\mathcal M}_1^1[k]$ generated by the 
$k$-th powers of Dehn twists (see \cite{FuKo,Mas1}). The following shows that 
the analog of Proposition \ref{noteq} does not hold:  
\begin{proposition}
If two matrices $A,B\in SL(2,\Z)={\mathcal M}_1^1$ 
are conjugate in each quotient ${\mathcal M}_1^1/{\mathcal M}_1^1[k]$
then $A$ and $B$ are conjugate in $SL(2,\Z)$.  
\end{proposition}
\begin{proof}
Let $F$ be a finite quotient of $SL(2,\Z)$. 
Then the image of the Dehn twist corresponding 
to a parabolic in $SL(2,\Z)$ is of finite order, say $k$. 
The Dehn twists on $\Sigma_{1}^1$ are conjugate so that 
$F$ is a quotient of ${\mathcal M}_1^1/{\mathcal M}_1^1[k]$. 
This implies that the images of $A$ and $B$
are conjugate in any finite quotient $F$. According to Stebe 
(see \cite{S}) the group $SL(2,\Z)$ is conjugacy separable 
and this means that $A$ and $B$ are conjugate. 
\end{proof}

\subsection{Equivalence relations on 3-manifolds}
There are some natural equivalence relations 
on the set of closed 3-manifolds which are inspired by the present 
constructions. 

At first there is Lackenby's congruence relation 
from the Introduction. 
Further two manifolds are  said Turaev-Viro equivalent 
if their Turaev-Viro invariants agree, for every spherical fusion category.  

The third equivalence relation is to consider a slight generalization of the 
transformations arising in Proposition \ref{TVagree}, which we will call local 
equivalence. Specifically we have 
an elementary locally equivalence between $M_1$ and $M_2$ if there exists 
a third  closed 3-manifold $N$ with a non-separating embedded 
2-torus $T\subset N$ and a pair of matrices $A_1$ and $A_2$ with the properties: 
\begin{enumerate}
\item The matrices $A_1$ and $A_2$ are locally equivalent, meaning
that they are conjugate in every congruence quotient. 
\item We obtain $M_i$ by cutting open $N$ along $T$ and gluing back 
the two torus components obtained after twisting by $A_i$. 
\end{enumerate}
Eventually $M$ and $N$ are called locally equivalent if 
there is a sequence of elementary local equivalences connecting $M$ and $N$. 

An easy extension of Proposition \ref{TVagree} shows that closed 3-manifolds 
which are locally equivalent are Turaev-Viro equivalent. 
On the other hand from \cite{Gil,L} one derive that 
congruent manifolds are also Turaev-Viro equivalent. 

It is not clear whether the three above relations are actually the same. 
It would be interesting to have examples of equivalent 
hyperbolic 3-manifolds, if they ever exist. 

The set of homeomorphisms types of torus bundles $M_B$ 
which are locally equivalent to $M_A$ is then a subset 
${\mathcal X}^{loc}(M)$ of ${\mathcal X}^{TV}(M)\subset {\mathcal I}(M)/\iota$. 
It is not clear a priori that all elements in ${\mathcal X}^{loc}(M)$ 
are of the form $M_B$ with $A$ locally equivalent to $B$. 
If true, this will permit to compute effectively the 
subset ${\mathcal X}^{loc}(M)$.

{
\small      
      
\bibliographystyle{plain}

}

\appendix
\section{Counting matrices in a given genus}
\vspace{0.5cm}
\begin{center}
by Louis Funar and Andrei Rapinchuk
\end{center}

\vspace{0.5cm}

For $t \in \mathbb{Z}$, we let
$$
\mathcal{M}_t = \{ A \in SL(2 , \mathbb{Z}) \: \vert \:
\mathrm{tr}(A) = t \},
$$
and let $\mathcal{X}_t$ denote the set of $GL(2 ,
\mathbb{Z})$-conjugacy classes of matrices in $\mathcal{M}_t$. We
define the {\it discriminant} of $A \in \mathcal{M}_t$ to be
$$
D= D(t) = \left\{ \begin{array}{cl} t^2 - 4 & \text{for} \ t \
\text{even}, \\ t^2/4 - 1 & \text{for} \ t \ \text{odd}.
\end{array} \right.
$$
Furthermore, the {\it genus} $\mathfrak{G}(A)$ of $A \in
\mathcal{M}_t$ is the set of $B \in SL(2 , \mathbb{Z})$ that are
conjugate to $A$ in $SL(2 , \widehat{\mathbb{Z}})$ where
$\widehat{\mathbb{Z}}$ is the profinite completion of $\mathbb{Z}$
(we note that obviously $\mathfrak{G}(A) \subset \mathcal{M}_t$).
Equivalently, $B \in \mathfrak{G}(A)$ if the images of $A$ and $B$
are conjugate in $SL(2 , \mathbb{Z}/m\mathbb{Z})$ for all $m > 1$.
It may appear that to comply with the general definition of genus
adopted in \cite[\S 8.5]{PlR} we would also need to require that $B$
must also be conjugate to $A$ in $SL(2 , \mathbb{Q})$, but here this
condition follows automatically from local conjugacy in view of the
Hasse norm theorem for quadratic extensions. On the other hand, one
can consider a variation of this definition of genus by requiring
that the the images of $A$ and $B$ in $SL(2 ,
\mathbb{Z}/m\mathbb{Z})$ be conjugate in $SL^{\pm}(2 ,
\mathbb{Z}/m\mathbb{Z})$, the group of matrices over
$\mathbb{Z}/m\mathbb{Z}$ with determinant $\pm 1$, for all $m > 1$;
the genus of $A$ thus defined will be denoted by
$\mathfrak{G}^{\pm}(A)$. Finally, we let $\mathfrak{A}(A)$ denote
the set of $SL(2 , \Z)$-conjugacy classes in $\mathfrak{G}(A)$.

Our main result is the following.

\begin{theorem}\label{lim}
There exists an increasing sequence of integers $\{ t_n \}$ such
that:

\begin{enumerate}
\item  $D_n := D(t_n)$ is square-free for all $n$;

\item We have: 
\begin{equation}
\max_{A\in \M_{t_n}}|\A(A)| \geqslant 0.1023\cdot 10^{-4}\cdot D_n^{0.49}
\frac{1}{2\log 2+ \log(D_n+2)}  
\end{equation} 
and therefore $\max_{A \in
\mathcal{M}_{t_n}} \vert \mathfrak{A}(A) \vert \longrightarrow
\infty$ as $n \to \infty$;
\item Eventually we have  
\begin{equation}
\displaystyle \frac{\max_{A \in \mathcal{M}_{t_n}} \vert
\mathfrak{A}(A) \vert}{\vert \mathcal{X}_{t_n} \vert} \geqslant
\frac{1}{64}.
\end{equation}
In particular,
\begin{equation}
\limsup_{t \to \infty} \frac{1}{\vert \mathcal{X}_t \vert} \max_{A
\in \mathcal{M}_{t}} \vert \mathfrak{A}(A) \vert \geqslant
\frac{1}{64}.
\end{equation}
\end{enumerate}
\end{theorem}

According to Propositions 1.1 and 1.2 above, matrices $A_1 , A_2 \in
\mathfrak{G}(A)$ such that neither of $A_1^{\pm 1}$ and $A_2^{\pm
1}$ are conjugate in $GL(2 , \mathbb{Z})$ (we will call such
matrices {\it strongly nonconjugate}) give rise to nonhomeomorphic
torus bundles having the same quantum invariants. This, in
particular, yields nonisomorphic 3-manifold groups having the same
profinite completion, answering the Grothendieck-type question
raised in \cite{LoR}. Theorem \ref{lim} above implies an asymptotic lower
bound on the size of a set of {\it pairwise} strongly nonconjugate
matrices in a genus inquired about in Remark 6.2 above, which gives
an effective version of Corollary 1.3. This effective version is closely 
related to the more general results of the 
second author  from \cite{PGR,PR2}. 
The proof below is based on 
the (well-known) connection between the conjugacy of
2-by-2 matrices and the equivalence of binary quadratic forms
(see \cite{Cassels}), although one can also give a direct argument.

\begin{proof}
Assume henceforth that $\vert t \vert \ge 3$
and set $D = D(t)$. First, we prove the following
result about the number of genera, which is based on the analysis of
local conjugacy (it should be noted that there are easy algorithms
to determine if two matrices in $SL(n , \mathbb{Z}_p)$ are
conjugate, for any $n$ (see \cite{AS}), but all we need for $n = 2$ is the
classical result about binary quadratic forms).

\begin{proposition}\label{genera}
Let 
$D=2^mp_1^{r_1}p_2^{r_2}\cdots p_n^{r_n}$ be the prime factorization 
of $D$.
\begin{enumerate}
\item 
The number of  distinct genera  $\G (A)$   contained in $\mathcal M_t$ is
\[s(t)=2^{n+\nu(D)}\cdot \tau(D)\]  
where $\tau(D)$ is the number of divisors of $D$ and 
\[ \nu(D)=\left\{\begin{array}{ll}  
0, & {\rm if }\, D\equiv 1 ({\rm mod} \; 2);  \\ 
0, & {\rm if }\, D=4d, d\equiv 1 ({\rm mod} \; 4);   \\
2, & {\rm if }\, D\equiv 0 ({\rm mod} \; 32);   \\
1, & {\rm otherwise}. \\
\end{array}\right.
\]
(note that $2^{n + \nu(D)}$ is the number of genera of primitive
binary quadratic forms of discriminant $D$).
\item 
The number of distinct genera  $\G ^{\pm}(A)$  in $\mathcal M_t$ is 
\[s_{\pm}(t)=2^{n_1+\nu(D)}\cdot \tau(D)\]  
where $\tau(D)$  and $\nu(D)$  are the same as above and 
$n_1$ is the number of odd prime factors $p_i\equiv 1 ({\rm mod } \; 4)$. 
\end{enumerate}
\end{proposition}
\begin{proof}
For a matrix $A = \left( \begin{array}{cc} a & b \\ c & d
\end{array} \right)$ which is not scalar one defines the {\it
Jorgensen invariant} to be
\[ J(A)= {\rm g.c.d.}(a-d, b,c)\]
As pointed out in \cite{Traina2}, $J(A)$ is an invariant of the conjugacy
class of $A$.

Furthermore, if one associates to the matrix $A$ the 
primitive bilinear form 
\[\frac{{\rm sgn}(tr(A))}{J(A)}\left( b x^2  - (a-d) xy - cy^2\right)\]
then conjugacy classes in $SL(2,\Z)$ will correspond to 
equivalence classes of bilinear forms.  

Then $J(A)$ can take $\tau(D)$ distinct values. Moreover the 
number of genera of primitive bilinear forms over $\Z$ 
was basically computed by Gauss (\cite{Gauss}),  
see (\cite{Cassels}, chap. 14, section 3, p.339--340, 
Lemmas 3.1-3.3) for a modern treatment. The case of improper equivalence 
classes is similar.     
\end{proof}

Denote by $\omega$ the element 
\begin{equation}
\omega= 
\left\{\begin{array}{ll}
\frac{1+ \sqrt{D}}{2}, & {\rm for } \; {\rm odd } \; t, \\
\sqrt{D},  & {\rm for } \; {\rm even } \; t, \\
\end{array}\right.
\end{equation}

According to the Latimer-MacDuffee-Taussky correspondence (see \cite{Newm})
there is a bijection between the elements of 
$\X_t$ and the ideal class group $\mathcal I(\Z[\omega])$ 
of the order $\Z[\omega]$. Denote then by $h(D)=|\mathcal I(\Z[\omega])|$ 
the class number of $\Z[\omega]$. Notice that $\Z[\omega]$ might not be 
the maximal order in $\Q(\sqrt{D})$ unless $D$ is square-free. 

Therefore there exists some $A\in \mathcal M_t$ 
such that the number of conjugacy classes in $\A(A)$ is at least:
\begin{equation}
N(t)=  2^{-s(t)} h(D)
\end{equation}

According to a celebrated theorem of Jing Run Chen (see \cite{Chen}) revisited 
by Halberstam (see \cite{Halb}) and Richert (see \cite{Richert}, Thm. 13.2) 
there exist infinitely many primes $p_n$ such that 
$p_n+4$ has at most two factor primes. Assuming that 
$p_n >5$ the two factor primes have to be distinct and different from $p_n$. 
If we set $t_n=p_n+2$ then $D_n=t_n^2-4$ are odd square-free and 
have at most 3 prime divisors (counted with their multiplicities).  
In particular $h(D_n)=h_{D_n}$ where this time $h_D$ 
denotes the class number of the quadratic field $\Q(\sqrt{D})$ 
(namely of its ring of integers). 

It remains to prove that for this subsequence we also 
have $\limsup h(D_n)=\infty$. This 
is already classical. Indeed the Dirichlet 
class number formula  for real quadratic fields reads: 

\begin{equation}
h_d=\frac{1}{2\log \epsilon_D} \sqrt{d} \cdot  L(1,\chi_d)
\end{equation}
where $d=4^{\delta_D}D$ is the discriminant of $\Q(\sqrt{D})$, 
$\epsilon_D$ is the fundamental unit, $\chi_d$ is 
the mod $d$ Dirichlet primitive character and $L(\cdot, \chi_d)$ the 
associated $L$-series. In our case $D_n\equiv 1\, ({\rm mod }\; 4)$ so 
that $\delta_{D_n}=0$ and the 
fundamental unit is $\epsilon_D=\frac{t+\sqrt{D}}{2}$, if  
$D=t^2-4$.  Thus $\epsilon_D < 2\sqrt{D+2}$. 

The Tatuzawa effective version of Siegel's theorem 
(see \cite{Tatuzawa}, Thm. 2)
states the following lower bound for the $L$-series: 
\begin{equation}
L(1,\chi_d) > 0.655\cdot \frac{s^{-1}}{d^{1/s}}
\end{equation}
for all  $d\geq \max(\exp(s), \exp(11.2))$  
with one possible exception and all $s > 2$.  
Eventually consider $s=100$ and $t_n$ large enough for which 
the inequality above holds. This gives our estimate.

\end{proof}
\begin{remark}
The congruence subgroup property implies that 
the estimates of Theorem \ref{lim} also hold in $SL(n,\Z)$, 
with $n\geq 3$, by  considering matrices of the form $A\oplus \mathbf 1_{n-2}$, 
with $A\in SL(2,\Z)$.
\end{remark}

{
\small      
      
\bibliographystyle{plain}

}

\end{document}